\documentclass[11pt]{amsart}
\usepackage{amssymb,amsmath,amsthm}
\usepackage{amsfonts}
\usepackage{enumitem}
\usepackage{dsfont}
\usepackage{cite}
\usepackage[colorlinks, citecolor=red]{hyperref}
\usepackage{mathrsfs}
\usepackage{epsfig}
\usepackage{lscape}
\usepackage{subfigure}
\usepackage{epstopdf}
\usepackage{caption}
\usepackage{algorithm}
\usepackage{algpseudocode}
\usepackage{multirow}
\usepackage{geometry}
\usepackage{url}
\usepackage{graphicx,cite}
\usepackage{longtable}
\usepackage{tikz}

\newtheorem{definition}{Definition}[section]

\newtheorem{theorem}[definition]{Theorem}
\newtheorem{lemma}[definition]{Lemma}

\newtheorem{remark}[definition]{Remark}

\newtheorem{example}[definition]{Example}

\newtheorem{assumption}{Assumption}[section]
\date{}

\begin{document}
	\baselineskip 18pt
	\bibliographystyle{plain}
	\title[RIM-PS for GAVE]{Randomized iterative methods with Polyak step-size for solving generalized absolute value equations}
	
	\author{Jiayun Chen}
	\address{School of Mathematical Sciences, Beihang University, Beijing, 100191, China. }
	\email{jiayunchen@buaa.edu.cn}
	
	\author{Qiye Zhang}
	\address{School of Mathematical Sciences, Beihang University, Beijing, 100191, China. }
	\email{zhangqiye@buaa.edu.cn}
	
	\author{Deren Han}
	\address{LMIB of the Ministry of Education, School of Mathematical Sciences, Beihang University, Beijing, 100191, China.}
	\email{handr@buaa.edu.cn}
	
		\author{Jiaxin Xie}
	\address{LMIB of the Ministry of Education, School of Mathematical Sciences, Beihang University, Beijing, 100191, China.}
	\email{xiejx@buaa.edu.cn}

	\begin{abstract}
		In this paper, we systematically incorporate the Polyak step-size into the randomized iterative method to improve its efficiency for solving generalized absolute value equations. In particular, we adopt the Polyak step-size within a stochastic iterative setting where the objective function updates dynamically at every step, unlike the classical Polyak step-size designed for deterministic optimization with fixed objective functions. Consequently, this novel implementation differs from the conventional Polyak scheme and demands a dedicated convergence analysis. We rigorously analyze the convergence properties of the proposed method and establish its linear convergence in expectation. Numerical experiments demonstrate that the incorporation of the Polyak step-size substantially improves the computational performance of randomized iterative methods with constant step-sizes.
	\end{abstract}
	
	\maketitle
	
	\let\thefootnote\relax\footnotetext{Key words: absolute value equations, randomized iterative method, Polyak step-size, linear convergence, Kaczmarz}
	
	\let\thefootnote\relax\footnotetext{Mathematics subject classification (2020): 15A06, 68W20, 90C30, 90C33}
	
	\section{Introduction}
	The generalized absolute value equation (GAVE)
	\begin{equation}\label{gave}
		Ax - B|x| = b,
	\end{equation}
	where $A,B\in \mathbb{R}^{m\times n}$, $b\in \mathbb{R}^{m}$, $|x|=(|x_1|, \dots, |x_n|)^\top$, and $\top$ denotes the transpose, has attracted increasing attention due to its relevance across multiple fields, such as the linear complementarity problem (LCP)~\cite{Mangasarian2006,Rohn2004}, biometrics~\cite{Dang2022}, game theory~\cite{Zhou2023}, and fluid mechanics~\cite{Brugnano2008}. Consequently, extensive research efforts have been devoted to analyzing the solvability~\cite{Hladik2023,Mangasarian2006,Wu2020,Wu2021,Xie2025,Kumar2024}, error bounds~\cite{Zamani2023,Xie2025,Li2025}, and iterative solvers~\cite{Cao2021,Rohn2014,Wang2019,Li2020,Yu2022,Zhou2021,Yu2024,Luo2026,Mangasarian2007,Alcantara2023,Xie2025,Chen2025,Chen2022,Li2016,Mangasarian2009,ChenYongxin2025,Li2023,Zhang2023,Dai2024,Mangasarian2007Concave} of the GAVE.
	
	In recent years, a wide variety of iterative approaches have been proposed for solving the GAVE, covering Newton-type methods~\cite{Cao2021,Mangasarian2009,Li2016,Wang2019,Yu2024,Chen2025}, Picard iterations~\cite{Rohn2014}, matrix splitting~\cite{Li2023,Zhou2021,Yu2024,Chen2022,ChenYongxin2025,Zhang2023,Dai2024}, as well as concave minimization-based strategies~\cite{Mangasarian2007,Mangasarian2007Concave}. For some further comments, we refer the reader to \cite{Hladik2026}. Nevertheless, nearly all the above-mentioned algorithms are limited to square coefficient matrices satisfying $m=n$. As far as we know, merely three solvers are available to tackle the non-square GAVE setting. The successive linearization algorithm (SLA) transforms the original problem into a concave minimization task and solves it through a finite series of linear programs~\cite{Mangasarian2007}. The method of alternating projections (MAP) recasts the absolute value equation as a two-set feasibility problem and proceeds by alternating projections onto the corresponding feasible regions~\cite{Alcantara2023}. Most recently, a flexible randomized iterative method, which unifies multiple classic algorithms as its special instances and enjoys linear convergence, is proposed in~\cite{Xie2025}. Randomized iterative methods retain the simplicity and computational efficiency of stochastic algorithms while attaining rapid convergence. However, such existing schemes adopt fixed constant step-sizes; see Section \ref{section2-3} for further details, and step-size selection is widely recognized to exert a crucial impact on the convergence performance of iterative algorithms~\cite{Yuan2008}. Recently, adaptive stepsize strategies have been developed for the randomized Kaczmarz method to handle inconsistent linear systems~\cite{Zeng2023RK,Zeng2025ASEIM}, while adaptive momentum techniques have also been proposed that automatically adjust parameters without requiring prior knowledge of the problem~\cite{Zeng2024ASHBM,han2024randomized}. These developments demonstrate the potential of dynamic stepsize and momentum adjustment within randomized frameworks. 
	
	The Polyak step-size \cite{Polyak1987} constitutes a classic adaptive rule free of hand-tuned parameters. It has garnered substantial attention within the optimization community for its remarkable convergence acceleration capability~\cite{Hutchinson2025,Wang2023,Loizou2021,Orvieto2022,Jiang2023,Zhang2025,Oikonomou2025}. In particular, when deployed for convex objectives with a known optimal function value, the Polyak step-size provably achieves optimality in convergence rate under certain regularity conditions~\cite{Polyak1987,Camerini2009,Brannlund1995}.
	The Polyak step-size has also been generalized to stochastic regimes, giving rise to the stochastic Polyak step-size (SPS)~\cite{Loizou2021}. Theoretical analysis establishes convergence guarantees for SPS over strongly convex, convex, and nonconvex objectives alike, and numerical validations confirm its practical superiority across a range of learning tasks \cite{Loizou2021}. Subsequent variants such as decreasing SPS~\cite{Orvieto2022} and adaptive SPS~\cite{Jiang2023} further remove the need for exact optimal function values and guarantee convergence to the exact solution under non-interpolation. The Polyak step-size has further been integrated with momentum-based acceleration techniques, including the heavy-ball method and gradient moving average~\cite{Wang2023}. Such hybrid schemes attain linear convergence and mitigate sensitivity to the momentum hyperparameter of the heavy-ball method~\cite{Wang2023,Oikonomou2025,Zhang2025}.
	
	In this paper, we systematically integrate the Polyak step-size into randomized iterative frameworks to enhance computational efficiency when solving the GAVE. Different from classical Polyak rules constructed for deterministic optimization with static objective functions, our proposed algorithm embeds the Polyak step-size into a stochastic iteration scheme. In this setting, the objective function is dynamically updated at each stochastic gradient descent (SGD) iteration. Such a novel construction deviates from existing Polyak variants and calls for a dedicated theoretical convergence study. Accordingly, we conduct a rigorous convergence analysis for the newly developed method and derive its expected linear convergence guarantee. Additionally, we perform extensive numerical experiments to validate that incorporating the Polyak step-size can effectively improve the computational performance of traditional randomized iterative methods with constant step-sizes.
	
	The remainder of the paper is organized as follows.
	After introducing some preliminaries in Section \ref{sec:prelim}, we present and analyze the randomized iterative method with Polyak step-size in Section \ref{sec:ARABK}. In Section~\ref{sec:numerical}, we perform some numerical experiments to show
	the effectiveness of the proposed method.
	Finally, we conclude the paper in Section \ref{sec:conclusion}.

	\section{Basic notations and preliminaries}\label{sec:prelim}
	
	\subsection{Basic notations}
	For a vector \( x \in \mathbb{R}^n \), we denote its \( i \)-th element, transpose, and Euclidean norm by \( x_i \), \( x^\top \), and \( \|x\|_2 \), respectively. The vector \( e_i \in \mathbb{R}^n \) represents the \( i \)-th standard basis vector, whose \( i \)-th element is one and all others are zero. We also write \( e = (1, 1, \dots, 1)^\top \in \mathbb{R}^n \). For a matrix \( A \in \mathbb{R}^{m \times n} \), we use \( A_{i,j} \), \( A_{i,:} \), and \( A_{:,j} \) to denote its \( (i,j) \)-th entry, \( i \)-th row, and \( j \)-th column, respectively. The transpose, spectral norm, Frobenius norm, and column space of \( A \) are denoted by \( A^\top \), \( \|A\|_2 \), \( \|A\|_F \), and \( \operatorname{Range}(A) \), respectively. We let \( \sigma_{\min}(A) \) and \( \sigma_{\max}(A) \) denote the smallest nonzero singular value and the largest singular value of \( A \), respectively, and \( \kappa_A \) denotes its condition number. If \( A \) is symmetric, we use \( \lambda_{\min}(A) \) and \( \lambda_{\max}(A) \) to denote its smallest and largest eigenvalues, respectively. 
	For an integer $m\geq 1$, let $[m]:=\{1,\ldots,m\}$.	Given $\mathcal{T}\subset[m]$, the cardinality
	of the set $\mathcal{T}$ is denoted by $\operatorname{card}(\mathcal{T})$. We use  $A_{\mathcal{T},:}$ to denote the row  submatrix indexed by $\mathcal{T} $.
	
	We denote the solution set of the GAVE by
	\[
	\mathcal{X}^* := \{ x \in \mathbb{R}^n : Ax - B|x| = b \}.
	\]
	The distance from \( x \) to \( \mathcal{X}^* \) is defined as
	\[
	\operatorname{dist}_{\mathcal{X}^*}(x) := \inf_{y \in \mathcal{X}^*} \|x - y\|_2.
	\]
	For random variables \( \xi_1 \) and \( \xi_2 \), we use \( \mathbb{E}[\xi_1] \) to denote the expectation of \( \xi_1 \), and \( \mathbb{E}[\xi_1 \mid \xi_2] \) to denote the conditional expectation of \( \xi_1 \) given \( \xi_2 \). If \( f \) is differentiable, we denote its gradient by \( \nabla f \).
	
	\subsection{Polyak step-size}
	\label{section2-2}
	
	This subsection briefly reviews the Polyak step-size~\cite{Polyak1987} and its stochastic variants~\cite{Loizou2021}. Consider the following unconstrained convex optimization problem
	\begin{equation}\label{opt_prob}
		\min_{x\in\mathbb{R}^n} f(x).
	\end{equation}
	To solve this problem, the gradient descent (GD) method takes the form
	\begin{equation}\label{GD}
		x^{k+1} = x^k - \alpha_k \nabla f(x^k),
	\end{equation}
	where $\alpha_k > 0$ is the step-size. If $f$ is $L$-smooth, that is,
	\[
	\|\nabla f(x) - \nabla f(y)\| \le L\|x-y\|,\quad \forall x,y\in\mathbb{R}^n,
	\]
	then a common choice for the step-size is $\alpha_k = 1/L$.

	However, in many practical scenarios, the objective function \( f \) may be nondifferentiable, or the smoothness parameter \( L \) may be unknown or difficult to estimate. In such cases, the Polyak step-size~\cite{Polyak1987} offers an elegant alternative that does not rely on \( L \). For any \( x \), we define the subgradient of \( f \) as
	\[
	\partial f(x) := \left\{ g \in \mathbb{R}^n \mid f(y) \ge f(x) + \langle g, y-x \rangle,\ \forall y \in \mathbb{R}^n \right\}.
	\]
	Given the update \( x^{k+1} = x^k - \alpha_k g_k \), the Polyak step-size is then defined as
	\begin{equation}\label{polyak_step}
		\alpha_k = \frac{f(x^k) - f^*}{\|g_k\|^2},
	\end{equation}
	where \( f^* := \min_x f(x) \) denotes the optimal value and \( g_k \in \partial f(x^k) \). In particular, if \( f \) is differentiable, the subgradient \( g_k \) reduces to the gradient \( \nabla f(x^k) \).
	
	In fact, the Polyak step-size can be derived from a simple analysis. From the update formula, we have
	\[
	\|x^{k+1} - x^*\|_2^2 = \|x^k - x^*\|_2^2 + \alpha^2 \|g_k\|^2 - 2\alpha \langle g_k, x^k - x^* \rangle.
	\]
	The minimizer of the right-hand side with respect to \( \alpha \) is \( \alpha_k = \langle g_k, x^k - x^* \rangle / \|g_k\|_2^2 \), but this quantity depends on the unknown optimum \( x^* \). To address this issue, we exploit convexity to obtain \( \langle g_k, x^k - x^* \rangle \ge f(x^k) - f^* \). Replacing the inner product by this lower bound gives exactly the Polyak step-size in \eqref{polyak_step}. This choice is particularly attractive as it requires no tuning, adapts to the local landscape of \( f \), and typically achieves strong practical performance.
	
	Recently, the Polyak step-size has been extended to the stochastic setting. Consider the finite-sum optimization problem
	\begin{equation}\label{finite_sum}
		\min_{x\in\mathbb{R}^n} f(x) = \frac{1}{n} \sum_{i=1}^n f_i(x).
	\end{equation}
	For large-scale problems with a huge number of samples \( n \), computing the full gradient \( \nabla f(x) = \frac{1}{n}\sum_{i=1}^n \nabla f_i(x) \) is computationally prohibitive. To alleviate this issue, the stochastic gradient descent (SGD) method is adopted, which randomly samples an index \( i_k \) at each iteration and executes the following update rule:
	\begin{equation}\label{SGD}
		x^{k+1} = x^k - \gamma_k \nabla f_{i_k}(x^k).
	\end{equation}
	When each component function \( f_i \) is \( L_i \)-smooth, the classical step-size \( \gamma_k = 1/L_{i_k} \) is commonly adopted. For the general case where each \( f_i \) is merely convex and differentiable, the stochastic Polyak step-size (SPS) was proposed in~\cite{Loizou2021} to solve the aforementioned finite-sum problem \eqref{finite_sum}, with the explicit form
	\begin{equation}\label{SPS}
		\gamma_k^{\mathrm{SPS}} = \frac{f_{i_k}(x^k) - f_{i_k}^*}{c \|\nabla f_{i_k}(x^k)\|^2},
	\end{equation}
	where \( i_k \) is the randomly selected index at iteration \( k \), \( f_{i_k}^* := \min_x f_{i_k}(x) \), and \( c > 0 \) is a constant. We refer to~\cite{oikonomou2025safeguarded} for more details on the recent developments of the Polyak step-size.
	
	\subsection{Randomized iterative method for GAVE}
	\label{section2-3}
	
	In this subsection, we briefly review the randomized iterative method (RIM) for GAVE proposed in~\cite{Xie2025}. The core idea of RIM is to incorporate randomization into the iterations through a user-defined distribution \(\mathcal{D}\), from which randomized sketching matrices \(S \in \mathbb{R}^{m \times \ell}\) are drawn, and single-step stochastic gradient descent (SGD) is employed at each iteration.
	
	Specifically, starting from an initial point \(x^0 \in \mathbb{R}^n\), at the \(k\)-th iteration with current iterate \(x^k\), we consider the following stochastic optimization problem:
	\begin{equation}\label{min}
		\min_{x \in \mathbb{R}^n} f^k(x) := \mathbb{E}_{S \sim \mathcal{D}} \left[ f_S^k(x) \right],
	\end{equation}
	where \(f_S^k(x) := \frac{1}{2} \| S^\top (Ax - B|x^k| - b) \|_2^2\) with \(S \in \mathbb{R}^{m \times \ell}\) being a random variable drawn from \(\mathcal{D}\).
	Different from conventional static optimization formulations, the constructed problem \eqref{min} possesses a dynamically varying objective function that depends explicitly on the current iterative state $x^k$.
	
	Starting from \( x^k \) and applying a single SGD step, one can obtain the following iteration scheme
	\begin{equation}\label{update}
		x^{k+1} = x^k - \alpha_k \nabla f_{S_k}^k(x^k) = x^k - \alpha_k A^\top S_k S_k^\top (A x^k - B|x^k| - b).
	\end{equation} 
	This is exactly the RIM algorithm proposed in~\cite{Xie2025}. In \cite{Xie2025}, the step-size is set to \( \alpha_k = \alpha / \|S_k^\top A\|_2^2 \) with \( \alpha \in (0,1] \). For a fixed sampled matrix \( S_k \), the function \( f_{S_k}^k \) defined in \eqref{min} is \( L \)-smooth with Lipschitz constant \( L_{S_k} = \|S_k^\top A\|_2^2 \). Accordingly, the adopted step-size can be rewritten as \( \alpha / L_{S_k} \), which essentially serves as a fixed constant step-size in each iteration. Motivated by the remarkable performance of the Polyak step-size in accelerating iterative optimization, we embed this adaptive step-size strategy into the aforementioned dynamic optimization framework, thereby developing a novel RIM equipped with adaptive Polyak step-size for solving GAVEs.
	
	\subsection{Useful assumptions and lemmas}
	Let \(\mathcal{D}\) be a distribution from which the sketching matrices \(S \in \mathbb{R}^{m \times \ell}\) are drawn. The expectation over \(\mathcal{D}\) is denoted by \(\mathbb{E}_{S \sim \mathcal{D}}[\cdot]\). Throughout this paper, we impose the following basic assumption on \(\mathcal{D}\).
	\begin{assumption}\label{ass:1}
		\begin{enumerate}[label=(\roman*)]
			\item The sampling space of \(\mathcal{D}\) is finite, and \(\mathbb{E}_{S \sim \mathcal{D}}[SS^\top]\) is positive definite.
			\item For every \(S \in \mathcal{D}\), the matrix \(A^\top S\) has full column rank, where \(A \in \mathbb{R}^{m \times n}\) is the coefficient matrix in \eqref{gave}.
		\end{enumerate}
	\end{assumption}
	
	Next, we show  that Assumption~\ref{ass:1} holds for a broad class of coefficient matrices $A$ and sketching distributions $\mathcal{D}$. 
	
	\begin{example}\label{example1}
		Let \( \mathcal{D} \) have the sampling space \( \{e_i\}_{i=1}^m \), with each \( i \) sampled with probability \( 1/m \). Assume that every row of \( A \) is nonzero.  Then \( \mathbb{E}_{S \sim \mathcal{D}}[SS^\top] = \frac{1}{m} I \) is positive definite, and \( A^\top S = A_{i,:}^\top \), which is a nonzero vector and therefore has full column rank.
	\end{example}

	
	\begin{example}\label{example2}
		We consider uniform sampling of $\ell$ distinct indices to construct the index set $\mathcal{I} \subset [m]$ with $\operatorname{card}(\mathcal{I}) = \ell$, where $\ell$ denotes the block size. The total number of candidate index sets is $\binom{m}{\ell}$, and each set $\mathcal{I}$ is sampled uniformly with probability $\operatorname{Prob}(\mathcal{I}) = 1 \big/ \binom{m}{\ell}$. The corresponding sketching matrix is defined as $S = I_{:, \mathcal{I}}$.
		Assume that any $\ell$ distinct rows of $A$ are linearly independent. Then $\mathbb{E}_{S\sim\mathcal{D}}[SS^\top] = \frac{\ell}{m}I$ is positive definite, and $A^\top S = A_{\mathcal{I},:}^\top$ possesses full column rank.
	\end{example}
	
	\begin{example}\label{example3}
		Consider a partition of \([m]\) given by
		\[
		\begin{aligned}
			\mathcal{I}_i &= \{\varpi(j) : j = (i - 1)\ell + 1, (i - 1)\ell + 2, \dots, i\ell\}, \quad i = 1, 2, \dots, t - 1,\\
			\mathcal{I}_t &= \{\varpi(j) : j = (t - 1)\ell + 1, (t - 1)\ell + 2, \dots, m\}, \quad \text{with } \operatorname{card}(\mathcal{I}_t) \leq \ell,
		\end{aligned}
		\]
		where \(\varpi\) is a uniform random permutation on \([m]\) and \(\ell\) is the block size. We define
		$
		\|A\|_{\varpi,\ell} := \sqrt{\sum_{i=1}^t \|A_{\mathcal{I}_i,:}\|_2^2},
		$
		and at each iteration select an index \(i_k \in [t]\) with probability
		$
		\operatorname{Prob}(i_k = i) = \frac{\|A_{\mathcal{I}_i,:}\|_2^2}{\|A\|_{\varpi,\ell}^2},
		$
		then set \(S = I_{:,\mathcal{I}_{i_k}}\). Assume that any \(\ell\) distinct rows of \(A\) are linearly independent. Then \(\mathbb{E}_{S \sim \mathcal{D}}[SS^\top]\) is positive definite, and \(A^\top S = A_{\mathcal{I},:}^\top\) has full column rank.
	\end{example}
	
	We note that the condition that all \(\ell\)-row subsets are linearly independent appears restrictive in Examples~\ref{example2} and~\ref{example3}, yet it holds almost surely when the rows of \(A\) are independent, zero-mean, isotropic Gaussian random vectors~\cite{Mityagin2020}.
	
	The following lemmas are crucial for our  analysis.

	\begin{lemma}[{\cite[Lemma 2.3]{Lorenz2025}}]\label{lem:H}
		Let $A\in\mathbb{R}^{m \times n}$ with $A \neq 0$, $S \in \mathbb{R}^{m \times \ell}$ and $\mathcal{D}$ be a distribution. Suppose Assumption~\ref{ass:1} holds. Then the matrix $H := \mathbb{E}\left[ \frac{S S^\top}{\|A^\top S\|_2^2} \right]$ is positive definite.
	\end{lemma}
	
	\begin{lemma}\label{lemma1}
		Assume \( A \in \mathbb{R}^{m \times n} \) and \( S \in \mathbb{R}^{m \times \ell} \) are such that \( A^\top S \) has full column rank. Then for any vector \( d \in \mathbb{R}^m \), it holds that \( S^\top d = 0 \) if and only if \( A^\top S S^\top d = 0 \).
	\end{lemma}
	\begin{proof}
		\( (\Rightarrow) \) If \( S^\top d = 0 \), multiplying both sides  by \( A^\top S \) directly yields \( A^\top S S^\top d = 0 \).
		
		\( (\Leftarrow) \) Since \( A^\top S \) has full column rank, its null space contains only the zero vector. Hence, if \( A^\top S S^\top d = 0 \), then \( S^\top d = 0 \).
	\end{proof}
		%
	
	\section{RIM with Polyak Step-Size for GAVE}\label{sec:ARABK}
	
	In this section, we develop a Polyak step-size variant of RIM for solving the GAVE~\eqref{gave}. We first introduce the step-size based on the discussions in Sections~\ref{section2-2} and~\ref{section2-3}. We then establish its linear convergence in expectation and provide conditions under which the convergence factor is strictly less than one.
	
	Recall from \eqref{update} that the standard RIM iteration is formulated as
	\[
	x^{k+1} = x^k - \alpha_k \nabla f_{S_k}^k(x^k) 
	= x^k - \alpha_k A^\top S_k S_k^\top \big( A x^k - B|x^k| - b \big).
	\]
	Inspired by the stochastic Polyak step-size in \eqref{SPS}, we design an adaptive iteration-dependent step-size
	\[
	\tilde{\alpha}_k = \frac{f_{S_k}^k(x^k) - (f_{S_k}^k)^*}{\|\nabla f_{S_k}^k(x^k)\|_2^2},
	\]
	where \((f_{S_k}^k)^* := \min_x f_{S_k}^k(x)\) denotes the minimum value of  
	$
	f_{S_k}^k(x) = \frac{1}{2} \| S_k^\top (A x - B|x^k| - b) \|_2^2.
	$
	Under Assumption~\ref{ass:1}, \(A^\top S_k\) has full column rank, which implies that the linear system
	$
	S_k^\top A x = S_k^\top (B|x^k| + b)
	$
	is consistent. Consequently, the optimal value satisfies \((f_{S_k}^k)^* = 0\). Hence we have
	\begin{equation}\label{beta}
		\tilde{\alpha}_k = \frac{1}{2} \cdot \frac{\| S_k^\top (A x^k - B|x^k| - b) \|_2^2}{\| A^\top S_k S_k^\top (A x^k - B|x^k| - b) \|_2^2}.
	\end{equation}
	Furthermore, we introduce a relaxation parameter \(\zeta \in (0,2)\) and define the final adaptive Polyak step-size as
	\begin{equation}\label{eq:alpha_k}
		\alpha_k =
		\begin{cases}
			(2 - \zeta) L_{\text{adap}}^{(k)}, & \text{if } S_k^\top (A x^k - B|x^k| - b) \neq 0, \\
			0, & \text{otherwise},
		\end{cases}
	\end{equation}
	where
	\[
	L_{\text{adap}}^{(k)} := \frac{\| S_k^\top (A x^k - B|x^k| - b) \|_2^2}{\| A^\top S_k S_k^\top (A x^k - B|x^k| - b) \|_2^2}.
	\]
	Under Assumption~\ref{ass:1}(ii), \(A^\top S_k\) has full column rank and thus, by Lemma~\ref{lemma1}, \(S_k^\top (A x^k - B|x^k| - b) \neq 0\) implies that \(A^\top S_k S_k^\top (A x^k - B|x^k| - b) \neq 0\), and hence \(\alpha_k\) is well-defined.
	We are now ready to present the Polyak step-size variant of RIM for solving the GAVE~\eqref{gave}, as described in Algorithm~\ref{alg:ARABK}.

	\begin{algorithm}[htpb]
		\caption{RIM with Polyak step-size (RIM-PS) for GAVE \label{alg:ARABK}}
		\begin{algorithmic}
			\Require
			$A, B \in \mathbb{R}^{m \times n},\; b \in \mathbb{R}^m$, distribution $\mathcal{D}$, relaxation parameters $\zeta \in (0, 2),\; k = 0$, and initial points $x^0 \in \mathbb{R}^n$.
			\begin{enumerate}
				\item[1:] Randomly select a sampling matrix $S_k \in \mathcal{D}$.
				\item[2:] Compute the step-size $\alpha_k$ using~\eqref{eq:alpha_k} then update
				\[
				x^{k+1} = x^k - \alpha_k A^\top S_k S_k^\top (Ax^k - B|x^k|-b).
				\]
				\item[3:] If a stopping rule is satisfied, stop and go to output. Otherwise, set $k = k+1$ and return to Step 1.
			\end{enumerate}
			
			\Ensure
			The approximate solution $x^k$.
		\end{algorithmic}
	\end{algorithm}
	
	\begin{remark}
		Let \( S \in \mathbb{R}^{m \times \ell} \) be the randomized sketching matrix sampled from the distribution \( \mathcal{D} \). When \( \ell = 1 \), \( S \) degenerates to a vector, so that \( S_k^\top (A x^k - B|x^k| - b) \) becomes a scalar. In this case, \( L_{\text{adap}}^{(k)} \) in \eqref{eq:alpha_k} simplifies to
		\[
		L_{\text{adap}}^{(k)} = \frac{\big( S_k^\top (A x^k - B|x^k| - b) \big)^2}{\| A^\top S_k \|_2^2 \cdot \big( S_k^\top (A x^k - B|x^k| - b) \big)^2}
		= \frac{1}{\| A^\top S_k \|_2^2}.
		\]
		Accordingly, the proposed Polyak step-size reduces to the constant step-size used in the original RIM method~\cite{Xie2025}. However, in the general case \( \ell \geq 2 \), the sketching matrix is no longer a vector, and the above simplification no longer holds. Consequently, \( L_{\text{adap}}^{(k)} \) varies at each iteration, leading to an adaptive step-size scheme that differs significantly from the constant step-size in standard RIM.
	\end{remark}
	
	\subsection{Convergence analysis}
	
	Let $\{S_k\}_{k \ge 0}$ be a sequence of independently and identically distributed (i.i.d.) random variables, with common distribution determined by the sampling strategy of Algorithm~\ref{alg:ARABK}. To describe the flow of information throughout the iterative process, we introduce the history $\mathcal{B}_k := (S_0, S_1, \dots, S_{k-1})$ for $k\ge 1$. For any random variable $X$, we employ $\mathbb{E}_k[X] := \mathbb{E}[X\mid\mathcal{B}_k]$ to denote the conditional expectation given the history up to step $k$. We also define
	\[
	\mathcal{Q}_k := \{ S_k \in \mathcal{D} \mid S_k^\top (A x^k - B|x^k| - b) \neq 0 \}.
	\]
	Evidently, the collection $\{\mathcal{Q}_k, \mathcal{Q}_k^c\}$ forms a partition of the sampling space of $\mathcal{D}$. 
	On the event $S_k \in \mathcal{Q}$, we introduce the corresponding conditional expectation $\mathbb{E}_{k, S_k \in \mathcal{Q}}[\cdot] := \mathbb{E}[\cdot \mid \mathcal{B}_k, S_k \in \mathcal{Q}]$.
	
	For the convergence analysis, we define
	\[ 
	\Lambda_{\min} := \inf_{S \in \mathcal{D},\, A^\top S \neq 0} \lambda_{\min}\left( \frac{A^\top S S^\top A}{\|A^\top S\|_2^2} \right).
	\]
	Since the sampling space of $\mathcal{D}$ is finite, we have $\Lambda_{\min}>0$. Moreover, note that $\frac{\|A^\top SS^\top A\|_2}{\|A^\top S\|_2^2}\leq 1$, we have $\Lambda_{\min}\leq 1$. For any $\epsilon>0$, let
	\begin{equation}\label{c}
		c_{\zeta, \epsilon} := 
		\begin{cases} 
			(2 - \zeta)(1 - (1 + \epsilon)(1 - \zeta)), & \text{if } \zeta \in (0, 1) \text{ and } \epsilon\in (0,\frac{\zeta}{1-\zeta}); \\ 
			1,&\text{if }\zeta=1; \\
			(2 - \zeta)(1 + (1 - \epsilon)(\zeta - 1)), & \text{if } \zeta \in (1, 2) \text{ and } \epsilon\in(0,\frac{\zeta}{\zeta-1}), 
		\end{cases}
	\end{equation}
	and
	\begin{equation}\label{d} 
		d_{\zeta, \epsilon} := 
		\begin{cases} 
			(2 - \zeta)\left(1 + (\epsilon^{-1} + 1)(1 - \zeta)\right), & \text{if } \zeta \in (0, 1) \text{ and } \epsilon\in(0,\frac{\zeta}{1-\zeta}); \\ 
			1,&\text{if } \zeta=1; \\
			(2 - \zeta)\left(1 + (\epsilon^{-1} - 1)(\zeta - 1)\right), & \text{if } \zeta \in(1, 2)\text{ and }\epsilon\in(0,\frac{\zeta}{\zeta-1}).
		\end{cases}
	\end{equation}
	Then we have the following result for the convergence of Algorithm~\ref{alg:ARABK}.
	
	\begin{theorem}\label{thm:convergence}
		Assume that the distribution $\mathcal{D}$ satisfies Assumption~\ref{ass:1} and GAVE~\eqref{gave} is solvable with nonempty solution set $\mathcal{X}^*$. Let $H := \mathbb{E}\left[ \frac{S S^\top}{\|A^\top S\|_2^2} \right]$ and $\{x^k\}_{k\ge0}$ be generated by Algorithm~\ref{alg:ARABK}. If $\zeta\in(0,1)$ and $\epsilon\in(0,\frac{\zeta}{1-\zeta})$, or $\zeta=1$, or $\zeta\in(1,2)$ and $\epsilon\in(0,\frac{\zeta}{\zeta-1})$, then
		\[
		\mathbb{E}\left[\operatorname{dist}_{\mathcal{X}^*}^2(x^{k+1})\right]
		\le
		\left(1 - c_{\zeta,\epsilon}\,\sigma_{\min}^2(H^{\frac{1}{2}}A) + \frac{d_{\zeta,\epsilon}}{\Lambda_{\min}}\|H^{\frac{1}{2}}B\|_2^2\right)
		\mathbb{E}\left[\operatorname{dist}_{\mathcal{X}^*}^2(x^k)\right],
		\]
		where $c_{\zeta,\epsilon}$ and $d_{\zeta,\epsilon}$ are defined in \eqref{c} and \eqref{d}, respectively.
	\end{theorem}
	\begin{proof}
		For any $x^k$, we use $x_k^*$ to denote the projection of $x^k$ onto the solution set $\mathcal{X}^*$, i.e.,
		$
		\|x^k - x_k^*\|_2 = \operatorname{dist}_{\mathcal{X}^*}(x^k).
		$
		When $S_k^\top \left( Ax^k - B|x^k|-b \right) \neq 0$, we have
		\begin{align*}
			\|x^{k+1} - x_k^*\|_2^2 &= \|x^k - \alpha_k A^\top S_k S_k^\top (Ax^k - B|x^k|-b) - x_k^*\|_2^2 \\
			&= \|x^k - x_k^*\|_2^2 + \alpha_k^2 \|A^\top S_k S_k^\top (Ax^k - B|x^k|-b)\|_2^2 \\
			&\quad - 2\alpha_k \langle S_k^\top A(x^k - x_k^*), S_k^\top (Ax^k - B|x^k|-b) \rangle \\
			&= \|x^k - x_k^*\|_2^2 + (2 - \zeta) \alpha_k \|S_k^\top (Ax^k - B|x^k|-b)\|_2^2 \\
			&\quad - 2\alpha_k \langle S_k^\top A(x^k - x_k^*), S_k^\top (Ax^k -B|x^k|-b) \rangle \\
			&= \|x^k - x_k^*\|_2^2 + (2 - \zeta) \alpha_k \|S_k^\top (Ax^k - B|x^k|-b)\|_2^2 \\
			&\quad + \alpha_k \|S_k^\top B (|x^k|-|x_k^*|)\|_2^2 - \alpha_k\|S_k^\top A(x^k - x_k^*)\|_2^2 \\
			&\quad - \alpha_k\|S_k^\top (Ax^k -B|x^k|- b)\|_2^2 \\
			&= \|x^k - x_k^*\|_2^2 - \alpha_k \|S_k^\top A(x^k - x_k^*)\|_2^2 + \alpha_k \|S_k^\top B(|x^k|-|x_k^*|)\|_2^2 \\
			&\quad - (\zeta - 1) \alpha_k \|S_k^\top (Ax^k -B|x^k|-b )\|_2^2,
		\end{align*}
		where the third equality follows from the definition of $\alpha_k$ and the forth equality is due to $-2\langle a, b \rangle=\|a-b\|_2^2-\|a\|_2^2-\|b\|_2^2$. For any $\epsilon>0$, we can get
		\begin{align*}
			\|S_k^\top (Ax^k -B|x^k|-b )\|_2^2 &= \|S_k^\top A(x^k-x_k^*)+S_k^\top (Ax_k^* -B|x^k|-b )\|_2^2 \\
			&= \|S_k^\top A(x^k - x_k^*) +  S_k^\top B(|x_k^*|-|x^k|)\|_2^2 \\
			&\leq (1 + \epsilon) \|S_k^\top A(x^k - x_k^*) \|_2^2  + (\epsilon^{-1} + 1) \|S_k^\top B(|x_k^*| - |x^k|) \|_2^2,
		\end{align*}
		where the inequality is due to $2\langle a, b \rangle =2\langle \sqrt{\epsilon}a, \epsilon^{-1/2}b \rangle \leq \epsilon \|a\|_2^2+\epsilon^{-1}\|b\|_2^2$.
		And
		\begin{align*}
			\|S_k^\top (Ax^k -B|x^k|-b )\|_2^2 \geq (1 - \epsilon)\|S_k^\top A(x^k -x_k^*)\|_2^2 - (\epsilon^{-1} - 1)\|S_k^\top B(|x_k^*|-|x^k|)\|_2^2.
		\end{align*}
		where the inequality is due to $2\langle a, b \rangle =2\langle \sqrt{\epsilon}a, \epsilon^{-1/2}b \rangle \geq -\epsilon \|a\|_2^2-\epsilon^{-1}\|b\|_2^2$.
		
		For the case that \(\zeta \in (0,1]\), we have
		\begin{align*}
			\|x^{k+1} - x_k^*\|_2^2 &\leq \|x^k - x_k^*\|_2^2 - (1 - (1 + \epsilon)(1 - \zeta))\alpha_k \left\|S_k^\top A(x^k - x_k^*)\right\|_2^2 \\
			&\quad + (1 + (\epsilon^{-1} + 1)(1 - \zeta))\alpha_k \|S_k^\top B(|x^k|-|x_k^*|)\|_2^2 \\
			&= \|x^k - x_k^*\|_2^2 - c_{\zeta,\epsilon}L_{\text{adap}}^{(k)}\|S_k^\top A(x^k - x_k^*)\|_2^2  + d_{\zeta,\epsilon}L_{\text{adap}}^{(k)}\|S_k^\top B(|x^k|-|x_k^*|)\|_2^2,
		\end{align*}
		where \(c_{\zeta,\epsilon} = (2-\zeta)(1 - (1 + \epsilon)(1 - \zeta))\) and \(d_{\zeta,\epsilon} = (2-\zeta)(1 + (\epsilon^{-1} + 1)(1 - \zeta))\).
		
		For the case that \(\zeta \in (1,2)\), we have
		\begin{align*}
			\|x^{k+1} - x_k^*\|_2^2 &\leq \|x^k -x_k^*\|_2^2 - (1 + (1-\epsilon)(\zeta-1)) \alpha_k \|S_k^\top A(x^k -x_k^*)\|_2^2 \\
			&\quad + (1 + (\epsilon^{-1}-1)(\zeta-1)) \alpha_k \|S_k^\top B(|x^k|-|x_k^*|)\|_2^2 \\
			&= \|x^k - x_k^*\|_2^2 - c_{\zeta,\epsilon} L_{\text{adap}}^{(k)} \|S_k^\top A(x^k - x_k^* )\|_2^2 + d_{\zeta,\epsilon} L_{\text{adap}}^{(k)} \|S_k^\top B(|x^k|-|x_k^*|) \|_2^2.
		\end{align*}
		where $c_{\zeta,\epsilon} = (2-\zeta)(1 + (1 - \epsilon)(\zeta-1))$ and $d_{\zeta,\epsilon} = (2-\zeta)(1 + (\epsilon^{-1} -1)(\zeta-1))$.

		Thus, for any $\zeta \in (0, 2)$, it holds that
		\begin{align*}
			&\mathbb{E}_{k, S_k \in \Omega} \left[ \|x^{k+1} - x_k^*\|_2^2 \right] \\
			&= \mathbb{P}(S_k \in \mathcal{Q}_k^c) \, \mathbb{E}_{k, S_k \in \mathcal{Q}_k^c} \left[ \|x^{k+1} - x_k^*\|_2^2 \right] 
			+ \mathbb{P}(S_k \in \mathcal{Q}_k) \, \mathbb{E}_{k, S_k \in \mathcal{Q}_k} \left[ \|x^{k+1} - x_k^*\|_2^2 \right] \\
			&\leq \mathbb{P}(S_k \in \mathcal{Q}_k^c) \, \mathbb{E}_{k, S_k \in \mathcal{Q}_k^c} \left[ \|x^k - x_k^*\|_2^2 \right] 
			+ \mathbb{P}(S_k \in \mathcal{Q}_k) \, \mathbb{E}_{k, S_k \in \mathcal{Q}_k} \left[ \|x^k - x_k^*\|_2^2 \right] \\
			&\quad - c_{\zeta, \epsilon} \, \mathbb{P}(S_k \in \mathcal{Q}_k) \, \mathbb{E}_{k, S_k \in \mathcal{Q}_k} 
			\left[ L^{(k)}_{\text{adap}} \|S_k^\top A(x^k - x_k^*)\|_2^2 \right] \\
			&\quad + d_{\zeta, \epsilon} \, \mathbb{P}(S_k \in \mathcal{Q}_k) \, \mathbb{E}_{k, S_k \in \mathcal{Q}_k} 
			\left[ L^{(k)}_{\text{adap}} \|S_k^\top B(|x^k|-|x_k^*|) \|_2^2 \right] \\
			&= \| x^k - x_k^* \|_2^2 
			- c_{\zeta,\epsilon} \underbrace{ P(S_k \in Q_k) \; \mathbb{E}_{k, S_k\in Q_k} \left[ L_{\mathrm{adap}}^{(k)} \| S_k^\top A(x^k - x_k^*) \|_2^2 \right]}_{\text{(a)}} \\
			&\quad + d_{\zeta,\epsilon} \underbrace{ P(S_k \in Q_k) \; \mathbb{E}_{k, S_k\in Q_k} \left[ L_{\mathrm{adap}}^{(k)} \| S_k^\top B(|x^k| - |x_k^*|) \|_2^2 \right] }_{\text{(b)}}
		\end{align*}
		
		Next, we analyze the two expressions $(a)$ and $(b)$, respectively. Since
		
		\[
		L_{\text{adap}}^{(k)} = \frac{\left\| S_k^\top \left( A x^k- B|x^k|-b \right) \right\|_2^2}{\left\| A^\top S_k S_k^\top \left( A x^k - B|x^k|-b \right) \right\|_2^2} \geq \frac{1}{\|A^\top S_k\|_2^2},
		\]
		we can establish a lower bound for the first expression as
		\begin{align*}
			(a)&\geq P(S_k \in \mathcal{Q}_k) \; \mathbb{E}_{k, S_k \in \mathcal{Q}_k} \left[ \frac{\left\| S_k^\top A(x^k - x_k^*) \right\|_2^2}{\|A^\top S_k\|_2^2} \right] \\
			&= \mathbb{E}_{k, S_k \in \Omega} \left[ \frac{\left\| S_k^\top A(x^k - x_k^*) \right\|_2^2}{\|A^\top S_k\|_2^2} \right] - P(S_k \in \mathcal{Q}_k^c) \; \mathbb{E}_{k, S_k \in \mathcal{Q}_k^c} \left[ \frac{\left\| S_k^\top A(x^k - x_k^*) \right\|_2^2}{\|A^\top S_k\|_2^2} \right] \\
			&\geq \sigma_{\min}^2(H^{\frac{1}{2}} A) \|x^k - x_k^*\|_2^2 - P(S_k \in \mathcal{Q}_k^c) \; \mathbb{E}_{k, S_k \in \mathcal{Q}_k^c} \left[ \frac{\left\| S_k^\top A(x^k - x_k^*) \right\|_2^2}{\|A^\top S_k\|_2^2} \right] \\
			&=\sigma_{\min}^2(H^{\frac{1}{2}} A) \|x^k - x_k^*\|_2^2 - P(S_k \in \mathcal{Q}_k^c) \; \mathbb{E}_{k, S_k \in \mathcal{Q}_k^c} \left[ \frac{\left\| S_k^\top B(|x^k| - |x_k^*|) \right\|_2^2}{\|A^\top S_k\|_2^2} \right],
		\end{align*}
		where the last equality is due to $S_k^\top (Ax^k-B|x^k|-b)=0$ for $S_k\in \mathcal{Q}_k^c$.
		
		In addition, since Assumption~\ref{ass:1}(ii), we have
		\begin{align*}
			L_{\text{adap}}^{(k)} &= \frac{\left\| S_k^\top \left(Ax^k-B|x^k|-b \right) \right\|_2^2}{\left\| A^\top S_k S_k^\top \left( A x^k - B|x^k|-b \right) \right\|_2^2} 
			\leq \frac{1}{\lambda_{\min} \left( \frac{A^\top S_k S_k^\top A}{\|A^\top S_k\|_2^2} \right) \|A^\top S_k\|_2^2} \leq \frac{1}{\Lambda_{\min} \|A^\top S_k\|_2^2},
		\end{align*}
		then the second expression can be bounded as
		\[(b) \leq \frac{1}{\Lambda_{\min}} \mathbb{P}(S_k \in Q_k) \mathbb{E}_{k, S_k \in Q_k} \left[ \frac{\|S_k^\top B(|x^k|-|x_k^*|) \|_2^2}{\|A^\top S_k\|_2^2} \right].\]
		By the assumptions of theorem, specifically, if $\zeta \in (0, 1)$ and $\epsilon \in (0, \frac{\zeta}{1-\zeta})$, or $\zeta = 1$, or $\zeta \in (1, 2)$ and $\epsilon \in (0, \frac{\zeta}{1-\zeta})$, it follows from \eqref{c} and \eqref{d} that $c_{\zeta,\epsilon} > 0$ and $d_{\zeta,\epsilon} > 0$.
		\begin{align*}
			\mathbb{E}_{k, S_k \in \Omega} [\|x^{k+1} - x_k^*\|_2^2] &\leq (1-c_{\zeta, \epsilon} \sigma_{\min}^2(H^{\frac{1}{2}} A)) \|x^k - x_k^*\|_2^2 \\
			&\quad + c_{\zeta, \epsilon} \mathbb{P}(S_k \in \mathcal{Q}_k^c) \mathbb{E}_{k, S_k \in \mathcal{Q}_k^c} 
			\left[ \frac{\|S_k^\top B(|x^k|-|x_k^*|) \|_2^2}{\|A^\top S_k\|_2^2} \right] \\
			&\quad + \frac{d_{\zeta, \epsilon}}{\Lambda_{\min}} \mathbb{P}(S_k \in \mathcal{Q}_k) \mathbb{E}_{k, S_k \in \mathcal{Q}_k} 
			\left[ \frac{\|S_k^\top B(|x^k|-|x_k^*|) \|_2^2}{\|A^\top S_k\|_2^2} \right] \\
			&\leq (1-c_{\zeta, \epsilon} \sigma_{\min}^2(H^{\frac{1}{2}} A)) \|x^k - x_k^*\|_2^2 \\
			&+ \frac{d_{\zeta, \epsilon}}{\Lambda_{\min}} \mathbb{E}_{k, S_k \in \Omega} 
			\left[ \frac{\|S_k^\top B(|x^k|-|x_k^*|) \|_2^2}{\|A^\top S_k\|_2^2} \right] \\
			&\leq (1-c_{\zeta, \epsilon}\sigma_{\min}^2(H^{\frac{1}{2}} A)+\frac{d_{\zeta, \epsilon}}{\Lambda_{\min}}\|H^\frac{1}{2}B\|_2^2) \|x^k -x_k^*\|_2^2,
		\end{align*}
		where the second inequality follows from the facts that $c_{\zeta, \epsilon} \leq d_{\zeta, \epsilon}$ and $\Lambda_{\min} \leq 1$. In total,
		\begin{equation}\label{inequality2}
			\mathbb{E}[\|x^{k+1} - x_k^*\|_2^2] 
			\leq (1-c_{\zeta, \epsilon} \sigma_{\min}^2(H^{\frac{1}{2}} A)+\frac{d_{\zeta, \epsilon}}{\Lambda_{\min}} \| H^{\frac{1}{2}} B \|_2^2)\mathbb{E} [\|x^k -x_k^*\|_2^2].
		\end{equation}
		By the definition of projection, $x_{k+1}^*$ is the closest point in $\mathcal{X}^*$ to $x^{k+1}$. Hence,
		\begin{equation}\label{inequality3}
			\operatorname{dist}_{\mathcal{X}^*}^2(x^{k+1})
			= \|x^{k+1} - x_{k+1}^*\|_2^2
			\le \|x^{k+1} - x_k^*\|_2^2.
		\end{equation}
		Combining \eqref{inequality2}, and \eqref{inequality3}, we have
		\[
		\mathbb{E}\left[\operatorname{dist}_{\mathcal{X}^*}^2(x^{k+1})\right]
		\le
		(1 - c_{\zeta,\epsilon} \sigma_{\min}^2(H^{\frac{1}{2}} A) + d_{\zeta,\epsilon} \frac{1}{\Lambda_{\min}} \| H^{\frac{1}{2}} B \|_2^2) \mathbb{E}\left[\operatorname{dist}_{\mathcal{X}^*}^2(x^k)\right].
		\]
		This complete the proof of this theorem.
	\end{proof}
	
	For the above expression to converge, the convergence factor must be less than 1. So we analyze two cases of $\zeta$ for the convergence factor.

		%
	
	\begin{theorem}\label{thm:contraction}
		Let $
		r(\zeta) := 1 - c_{\zeta,\epsilon}\,\sigma_{\min}^2(H^{\frac{1}{2}}A) + \frac{d_{\zeta,\epsilon}}{\Lambda_{\min}}\|H^{\frac{1}{2}}B\|_2^2
		$
		be the convergence factor in Theorem~\ref{thm:convergence}. Then the following statements hold:
		\begin{itemize}
			\item[(i)] If \(\zeta \in (0,1]\), then \(r(\zeta) < 1\) provided that
			\[
			\zeta > \frac{\epsilon\,\Lambda_{\min}\sigma_{\min}^2(H^{\frac{1}{2}}A) + (\epsilon^{-1} + 2)\|H^{\frac{1}{2}}B\|_2^2}
			{(1 + \epsilon)\Lambda_{\min}\sigma_{\min}^2(H^{\frac{1}{2}}A) + (\epsilon^{-1} + 1)\|H^{\frac{1}{2}}B\|_2^2}.
			\]
			
			\item[(ii)] If \(\zeta \in (1,2)\), then \(r(\zeta) < 1\) provided that one of the following conditions holds:
			\begin{enumerate}[label=(\alph*)]
				\item \((1 - \epsilon)\Lambda_{\min}\sigma_{\min}^2(H^{\frac{1}{2}}A) > (\epsilon^{-1} - 1)\|H^{\frac{1}{2}}B\|_2^2\) and
				\[
				\zeta > \frac{(2 - \epsilon^{-1})\|H^{\frac{1}{2}}B\|_2^2 - \epsilon\,\Lambda_{\min}\sigma_{\min}^2(H^{\frac{1}{2}}A)}
				{(1 - \epsilon)\Lambda_{\min}\sigma_{\min}^2(H^{\frac{1}{2}}A) - (\epsilon^{-1} - 1)\|H^{\frac{1}{2}}B\|_2^2};
				\]
				
				\item \((1 - \epsilon)\Lambda_{\min}\sigma_{\min}^2(H^{\frac{1}{2}}A) < (\epsilon^{-1} - 1)\|H^{\frac{1}{2}}B\|_2^2\) and
				\[
				\zeta < \frac{(2 - \epsilon^{-1})\|H^{\frac{1}{2}}B\|_2^2 - \epsilon\,\Lambda_{\min}\sigma_{\min}^2(H^{\frac{1}{2}}A)}
				{(1 - \epsilon)\Lambda_{\min}\sigma_{\min}^2(H^{\frac{1}{2}}A) - (\epsilon^{-1} - 1)\|H^{\frac{1}{2}}B\|_2^2};
				\]
				
				\item \((1 - \epsilon)\Lambda_{\min}\sigma_{\min}^2(H^{\frac{1}{2}}A) = (\epsilon^{-1} - 1)\|H^{\frac{1}{2}}B\|_2^2\) and
				\[
				(2 - \epsilon^{-1})\|H^{\frac{1}{2}}B\|_2^2 < \epsilon\,\Lambda_{\min}\sigma_{\min}^2(H^{\frac{1}{2}}A).
				\]
			\end{enumerate}
		\end{itemize}
	\end{theorem}
	
	\begin{proof}
		For the case where \(\zeta \in (0,1]\), we have
		\[
		c_{\zeta,\epsilon} = (2-\zeta)\big(1 - (1+\epsilon)(1-\zeta)\big) = (2-\zeta)\big((1+\epsilon)\zeta - \epsilon\big)
		\]
		and
		\[
		d_{\zeta,\epsilon} = (2-\zeta)\big(1 + (\epsilon^{-1}+1)(1-\zeta)\big) = (2-\zeta)\big((\epsilon^{-1}+2) - (\epsilon^{-1}+1)\zeta\big).
		\]
		Since \(2-\zeta > 0\), the condition \(r(\zeta) < 1\) is equivalent to
		\begin{equation}\label{xie-0721-1}
			c_{\zeta,\epsilon} \sigma_{\min}^2(H^{\frac{1}{2}}A) > d_{\zeta,\epsilon} \frac{1}{\Lambda_{\min}} \| H^{\frac{1}{2}} B \|_2^2,
		\end{equation}
		which, after substituting the expressions for \(c_{\zeta,\epsilon}\) and \(d_{\zeta,\epsilon}\), becomes
		\[
		\big((1+\epsilon)\zeta - \epsilon\big) \sigma_{\min}^2(H^{\frac{1}{2}}A)
		>
		\big((\epsilon^{-1}+2) - (\epsilon^{-1}+1)\zeta\big) \frac{1}{\Lambda_{\min}} \| H^{\frac{1}{2}} B \|_2^2.
		\]
		Rearranging terms yields
		\[
		\left( (1+\epsilon)\sigma_{\min}^2(H^{\frac{1}{2}}A) + (\epsilon^{-1}+1)\frac{1}{\Lambda_{\min}} \| H^{\frac{1}{2}} B \|_2^2 \right) \zeta
		>
		\epsilon \, \sigma_{\min}^2(H^{\frac{1}{2}}A) + (\epsilon^{-1}+2)\frac{1}{\Lambda_{\min}} \| H^{\frac{1}{2}} B \|_2^2.
		\]
		Multiplying both sides by \(\Lambda_{\min} > 0\), we obtain
		\[
		\zeta > \frac{\epsilon\,\Lambda_{\min}\sigma_{\min}^2(H^{\frac{1}{2}}A) + (\epsilon^{-1}+2)\| H^{\frac{1}{2}} B \|_2^2}
		{(1+\epsilon)\Lambda_{\min}\sigma_{\min}^2(H^{\frac{1}{2}}A) + (\epsilon^{-1}+1)\| H^{\frac{1}{2}} B \|_2^2},
		\]
		which is exactly the condition stated in Theorem~\ref{thm:contraction}(i).
		
		For the case where \(\zeta \in (1,2)\), we have
		\[
		c_{\zeta,\epsilon} = (2-\zeta)\left( 1 + (1-\epsilon)(\zeta-1)\right) = (2-\zeta)\left((1-\epsilon)\zeta + \epsilon\right)
		\]
		and
		\[
		d_{\zeta,\epsilon} = (2-\zeta)\left(1 + (\epsilon^{-1}-1)(\zeta-1)\right) = (2-\zeta)\left( (\epsilon^{-1}-1)\zeta + (2-\epsilon^{-1})\right).
		\]
		Since \(2-\zeta > 0\), after substituting the expressions for \(c_{\zeta,\epsilon}\) and \(d_{\zeta,\epsilon}\) into \eqref{xie-0721-1}, the condition \(r(\zeta) < 1\) is equivalent to
		\[
		\left((1-\epsilon)\zeta + \epsilon\right) \sigma_{\min}^2(H^{\frac{1}{2}} A) > \left( (\epsilon^{-1}-1)\zeta + (2-\epsilon^{-1})\right) \frac{1}{\Lambda_{\min}} \| H^{\frac{1}{2}} B \|_2^2.
		\]
		Rearranging terms yields
		\[
		\left( (1-\epsilon)\sigma_{\min}^2(H^{\frac{1}{2}} A) - (\epsilon^{-1}-1)\frac{1}{\Lambda_{\min}} \| H^{\frac{1}{2}} B \|_2^2\right) \zeta > (2-\epsilon^{-1})\frac{1}{\Lambda_{\min}} \| H^{\frac{1}{2}} B \|_2^2 - \epsilon \sigma_{\min}^2(H^{\frac{1}{2}} A).
		\]
		We verify this inequality by examining the following subcases.
		
		\noindent\textbf{Case 2.1.} If 
		$
		(1-\epsilon)\sigma_{\min}^2(H^{\frac{1}{2}} A) - (\epsilon^{-1}-1)\frac{1}{\Lambda_{\min}} \| H^{\frac{1}{2}} B \|_2^2 > 0,
		$
		then
		\[
		\zeta > \frac{(2-\epsilon^{-1}) \| H^{\frac{1}{2}} B \|_2^2 - \epsilon\Lambda_{\min} \sigma_{\min}^2(H^{\frac{1}{2}} A)}
		{(1-\epsilon)\Lambda_{\min}\sigma_{\min}^2(H^{\frac{1}{2}} A) - (\epsilon^{-1}-1) \| H^{\frac{1}{2}} B \|_2^2}.
		\]
		
		\noindent\textbf{Case 2.2.} If 
		$
		(1-\epsilon)\sigma_{\min}^2(H^{\frac{1}{2}} A) - (\epsilon^{-1}-1)\frac{1}{\Lambda_{\min}} \| H^{\frac{1}{2}} B \|_2^2 < 0,
		$
		then
		\[
		\zeta < \frac{(2-\epsilon^{-1}) \| H^{\frac{1}{2}} B \|_2^2 - \epsilon\Lambda_{\min} \sigma_{\min}^2(H^{\frac{1}{2}} A)}
		{(1-\epsilon)\Lambda_{\min}\sigma_{\min}^2(H^{\frac{1}{2}} A) - (\epsilon^{-1}-1) \| H^{\frac{1}{2}} B \|_2^2}.
		\]
		
		\noindent\textbf{Case 2.3.} If 
		$
		(1-\epsilon)\sigma_{\min}^2(H^{\frac{1}{2}} A) - (\epsilon^{-1}-1)\frac{1}{\Lambda_{\min}} \| H^{\frac{1}{2}} B \|_2^2 = 0,
		$
		then the inequality reduces to
		$
		(2-\epsilon^{-1})\frac{1}{\Lambda_{\min}} \| H^{\frac{1}{2}} B \|_2^2 - \epsilon \sigma_{\min}^2(H^{\frac{1}{2}} A) < 0,
		$
		or equivalently,
		\[
		(2-\epsilon^{-1}) \| H^{\frac{1}{2}} B \|_2^2 < \epsilon\,\Lambda_{\min} \sigma_{\min}^2(H^{\frac{1}{2}} A).
		\]
		Combining all the cases analyzed above yields the conditions for the convergence factor \( r(\zeta) \) stated in Theorem~\ref{thm:contraction}.
	\end{proof}

	\section{Numerical experiments}\label{sec:numerical}
	
	In this section, we present numerical experiments to demonstrate the efficiency of the proposed RIM-PS method for solving GAVE \eqref{gave}. We also compare its performance against several state-of-the-art approaches, including SLA \cite{Mangasarian2007}, MAP \cite{Alcantara2023}, and RABK \cite{Xie2025}.
	All experiments are conducted on a Lenovo ThinkPad S2 Yoga laptop equipped with an Intel(R) Core(TM) i5-10210U CPU @ 1.60GHz and 8 GB of RAM, running Windows 11 and MATLAB R2022a.
	
	The coefficient matrices $A$ and $B$ in \eqref{gave} are generated via the compact SVD factorization
	\[
	A = U_A \Sigma_A V_A^T, \qquad B = U_B \Sigma_B V_B^T,
	\]
	where $U_A, U_B \in \mathbb{R}^{m \times r}$ and $V_A, V_B \in \mathbb{R}^{n \times r}$ are column-orthogonal matrices, with $r := \min\{m,n\}$. The diagonal entries of $\Sigma_A$ and $\Sigma_B$ are given by
	\[
	(\Sigma_A)_{i,i} = a_{\min} + \frac{i-1}{r-1}(\kappa_A - 1)a_{\min},
	\qquad
	(\Sigma_B)_{i,i} = \frac{b_{\max}}{\kappa_B} + \frac{i-1}{r-1}\left(1 - \frac{1}{\kappa_B}\right)b_{\max},
	\]
	for $i = 1, \dots, r$, where $a_{\min}, b_{\max} > 0$ and $\kappa_A, \kappa_B \ge 1$ are user-specified parameters.
	This construction ensures that the diagonal entries of $\Sigma_A$ and $\Sigma_B$ range linearly from $a_{\min}$ to $b_{\max}$. In particular,  $\sigma_{\min}(A) > \sigma_{\max}(B)$ can guarantee the unique solvability of GAVE~\cite{Rohn2009}. Using MATLAB notation, we generate each of $U_A, V_A, U_B$, and $V_B$ by setting ${\tt [U,\sim]=\mathrm{qr}(\mathrm{randn}(m,r),0)}$ and ${\tt [V,\sim]=\mathrm{qr}(\mathrm{randn}(n,r),0)}$. In all experiments, we fix $a_{\min}  = 2$ and $b_{\max} = 1$. The solution is generated by $x^* =  {\tt randn(n, 1)}$ and $b$ is computed as $b=Ax^*-B|x^*|$. To evaluate the convergence performance of algorithms, we adopt three criteria: the relative solution error (RSE), the number of iterations and the CPU time. The ``RSE'' is defined as $\text{RSE}=\|x^k-x^*\|_2^2/\|x^*\|_2^2$.
	When RSE reaches the set threshold, the computations are terminated. The number of iterations indicates how many steps are required to achieve convergence, and the CPU time directly indicates the overall computational efficiency.
	
	In our experiments, we consider three sketching strategies: uniform sampling, partition sampling, and Gaussian sampling. The first two strategies are described in Example~\ref{example2} and Example~\ref{example3}, respectively. For Gaussian sampling, the sketching matrix \(S \in \mathbb{R}^{m \times \ell}\) has i.i.d. entries drawn from the standard normal distribution \(\mathcal{N}(0,1)\). We note that Gaussian sampling requires multiplying the full matrix by a dense Gaussian matrix, which is computationally far more expensive than the other two strategies, as the latter involve only the selected blocks in matrix products. The partition sampling strategy coincides with that used in the randomized average block Kaczmarz (RABK) method~\cite{Xie2025,Du2020,Necoara2019,Moorman2021}. Accordingly, we denote the three variants of RIM-PS as RABK-PS (partition), RIMUS-PS (uniform), and RIMGS-PS (Gaussian). All numerical results are averaged over 20 independent runs to mitigate the effects of randomness, and we fix \(\zeta = 1\) throughout the experiments.
	%
	
	\subsection{Comparison of different sampling strategies and block sizes}
	\label{sec:comparison_sketches}
	
	In this subsection, we first compare the computational performance of RABK-PS, RIMUS-PS, and RIMGS-PS. Figure~\ref{fig:sketch_comparison} (a)-(c) depicts the relationship between RSE and CPU time for the three adaptive methods under three different combinations of $\kappa_A$ and $\kappa_B$. In each subfigure, the bold line illustrates the median, the lightly shaded area signifies the range from the minimum to the maximum, and the darker shaded area indicates the data lying between the 25-th and 75-th quantiles.
	
	\begin{figure}[htbp]
		\centering
		\subfigure[$\kappa_A=2$, $\kappa_B=1$]{%
			\includegraphics[width=0.32\textwidth]{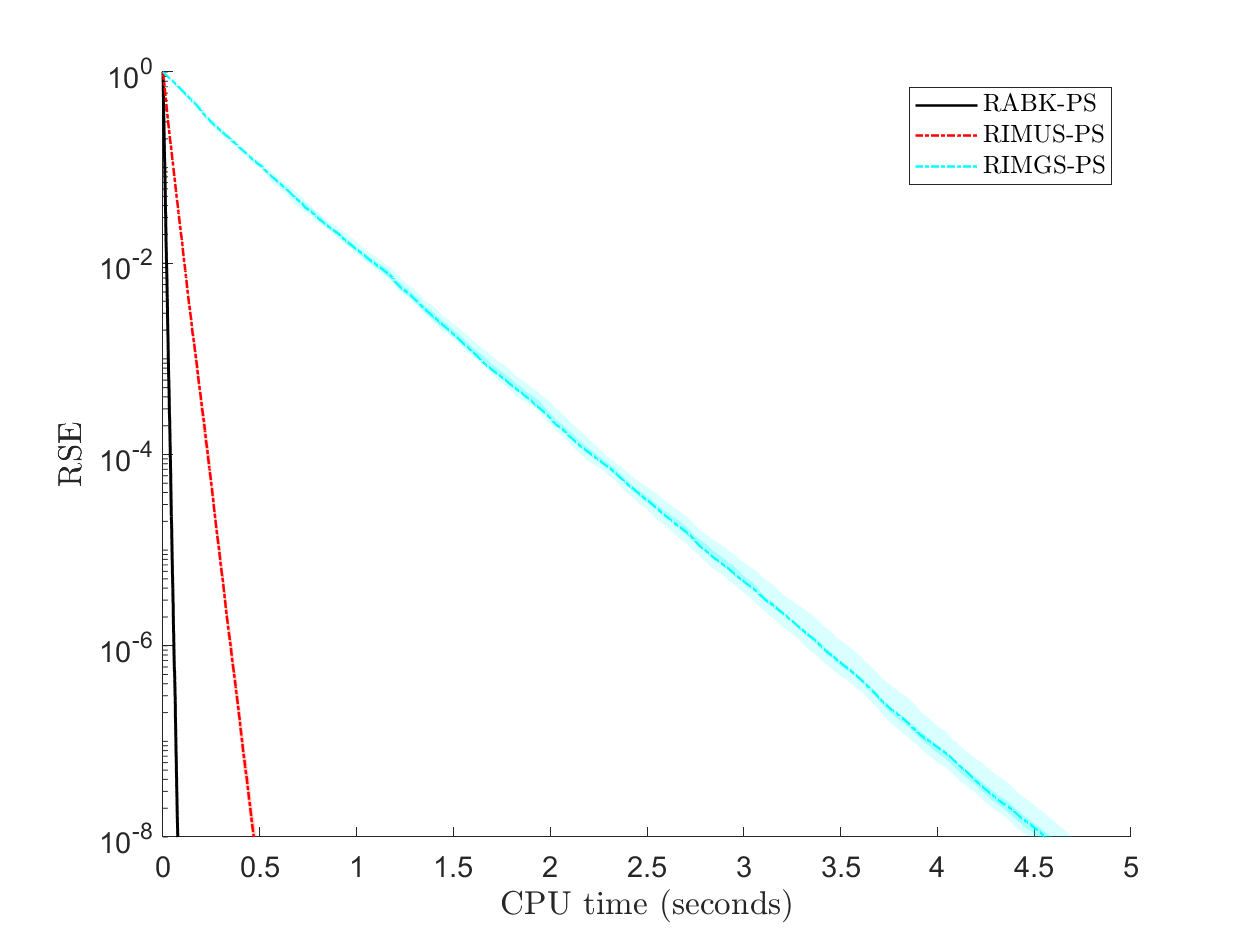}%
			\label{fig:conv1}%
		}%
		\hfill
		\subfigure[$\kappa_A=2$, $\kappa_B=10$]{%
			\includegraphics[width=0.32\textwidth]{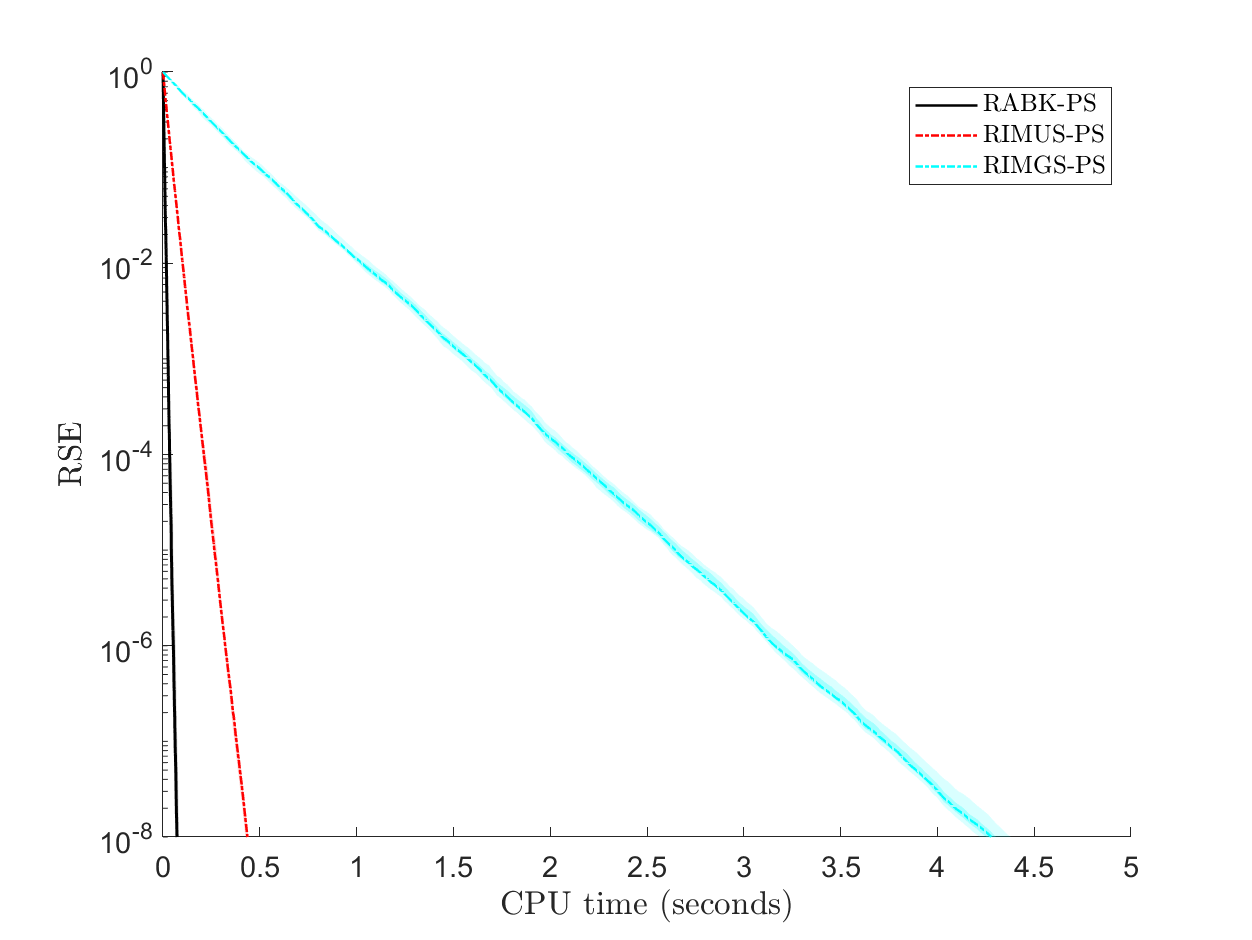}%
			\label{fig:conv2}%
		}%
		\hfill
		\subfigure[$\kappa_A=10$, $\kappa_B=1$]{%
			\includegraphics[width=0.32\textwidth]{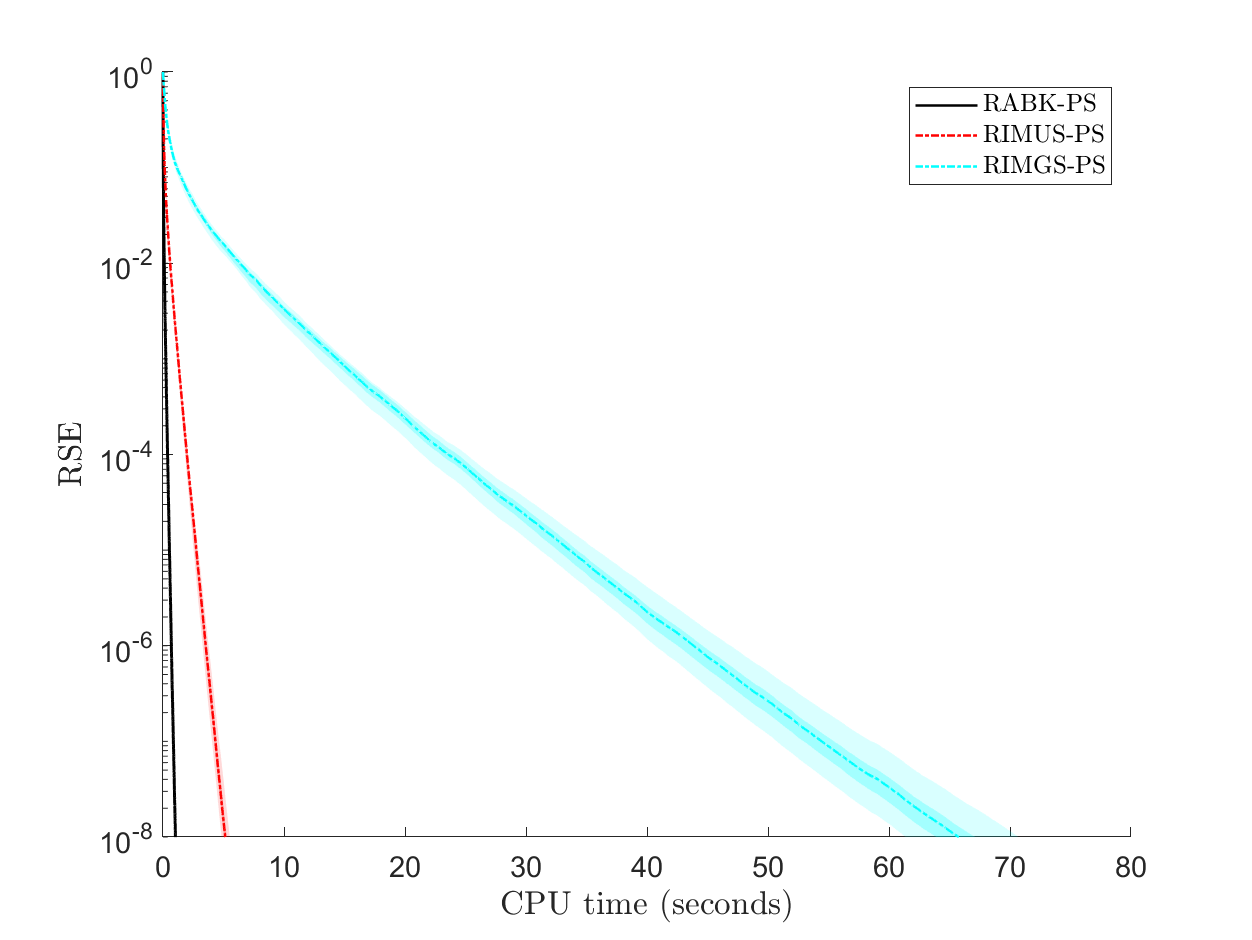}%
			\label{fig:conv3}%
		}
		\caption{Comparison of RABK-PS, RIMUS-PS, and RIMGS-PS. We set $m=2000$, $n=1000$ and $\ell=100$. Figures (a)-(c) depict the evolution of RSE with respect to CPU time under three different conditions.}
		\label{fig:sketch_comparison}
	\end{figure}
	
	It can be observed  in Figure~\ref{fig:sketch_comparison} that  RABK-PS requires the least CPU time among the three methods across all prescribed accuracies, followed by RIMUS-PS, while RIMGS-PS is the most time-consuming. The CPU time of all methods increases with the growth of $\kappa_A$, whereas $\kappa_B$ exhibits only a marginal effect on the computational cost. Given its superior performance, we adopt RABK-PS as the representative method in the subsequent experiments to investigate the influence of the block size.
	
	To assess the impact of the block size $\ell$, we vary it from $1$ to $512$ (powers of two), with $m=512$ and $n=128, 256, 512$. For each setting, we record the iteration count $k$ and report both the CPU time and the equivalent number of full passes, $k \cdot (\ell / m)$, to enable fair comparisons across different block sizes~\cite{Xie2025}.
	Figure~\ref{fig:block_size_influence} shows the number of full iterations and the CPU time as functions of \(\log_{2}(\ell)\). In all cases, the CPU time decreases sharply when moving from the single-row variant (\(\ell=1\)) to a moderately sized block. In contrast to the CPU time, the number of full iterations \(k \cdot \ell / m\) grows monotonically with \(\ell\) for nearly all tested problem sizes and condition numbers. This reflects a fundamental trade-off: increasing the block size reduces the number of iterations \(k\) and improves cache utilization and parallelization, but also raises the computational cost per iteration. The CPU time reaches its minimum when these two opposing effects are balanced, which explains the eventual increase in CPU time for very large block sizes.

	\begin{figure}[htbp]
		\centering
		\subfigure[$\kappa_A=2$, $\kappa_B=10$]{%
			\includegraphics[width=0.32\textwidth]{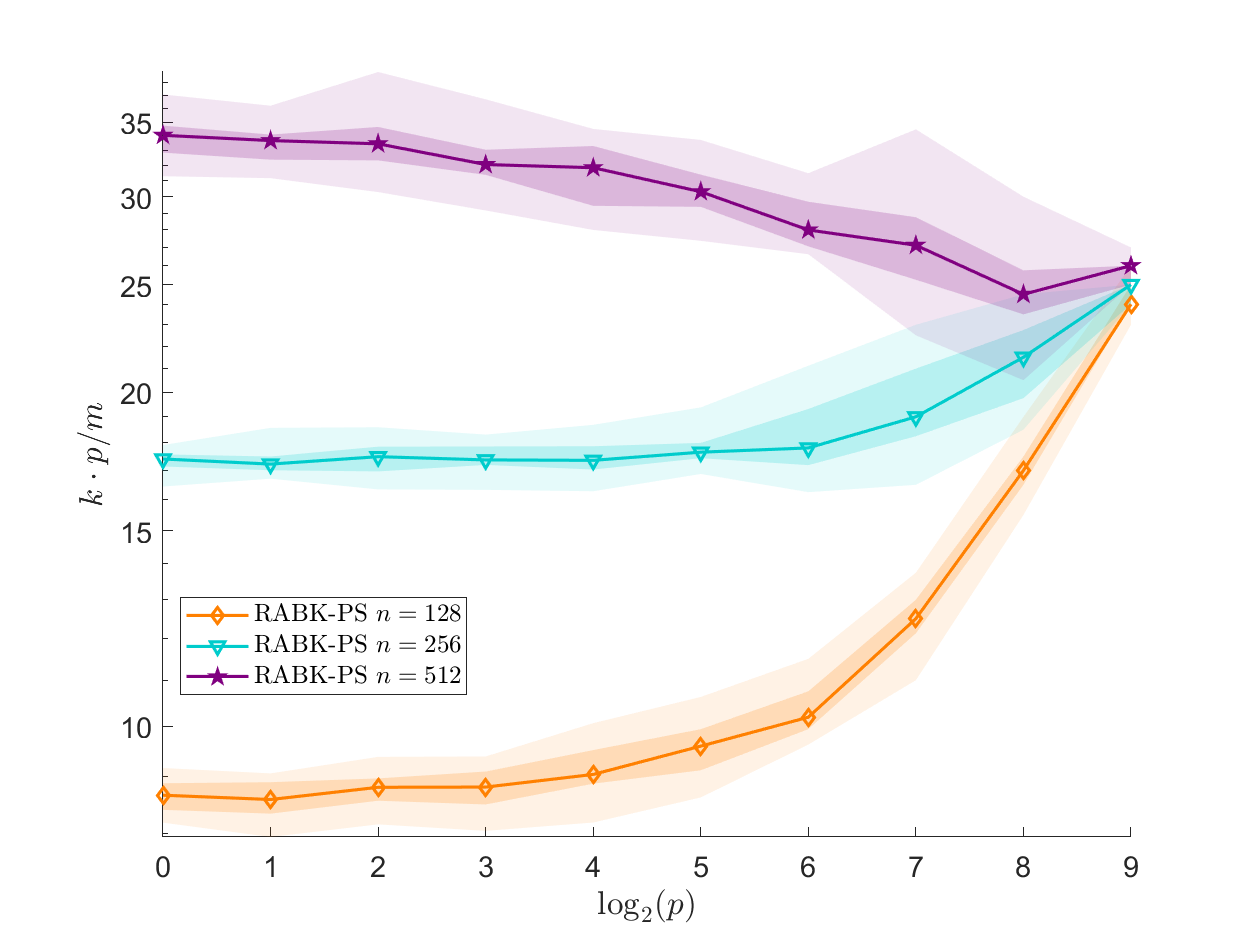}%
		}%
		\hfill
		\subfigure[$\kappa_A=10$, $\kappa_B=10$]{%
			\includegraphics[width=0.32\textwidth]{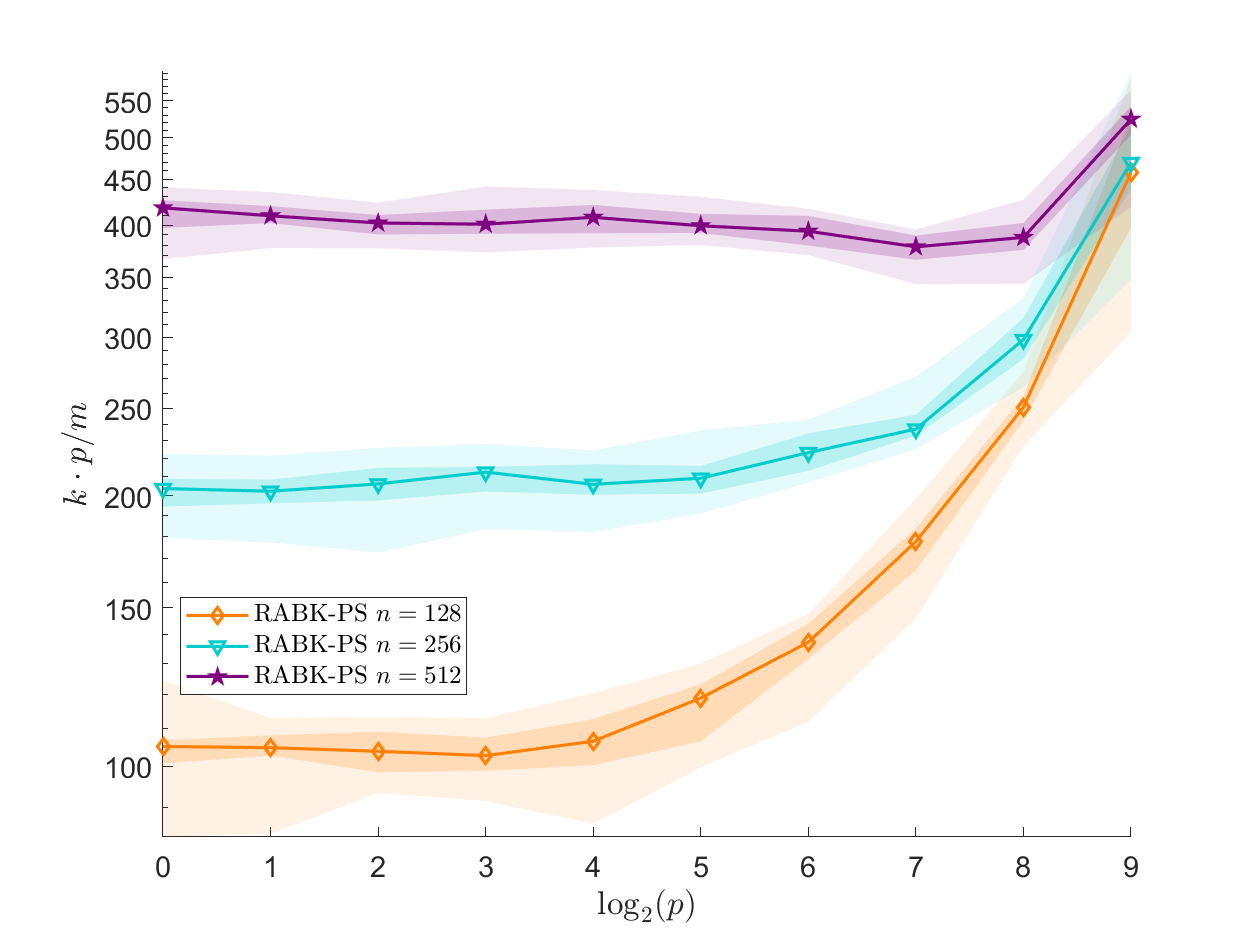}%
		}%
		\hfill
		\subfigure[$\kappa_A=10$, $\kappa_B=2$]{%
			\includegraphics[width=0.32\textwidth]{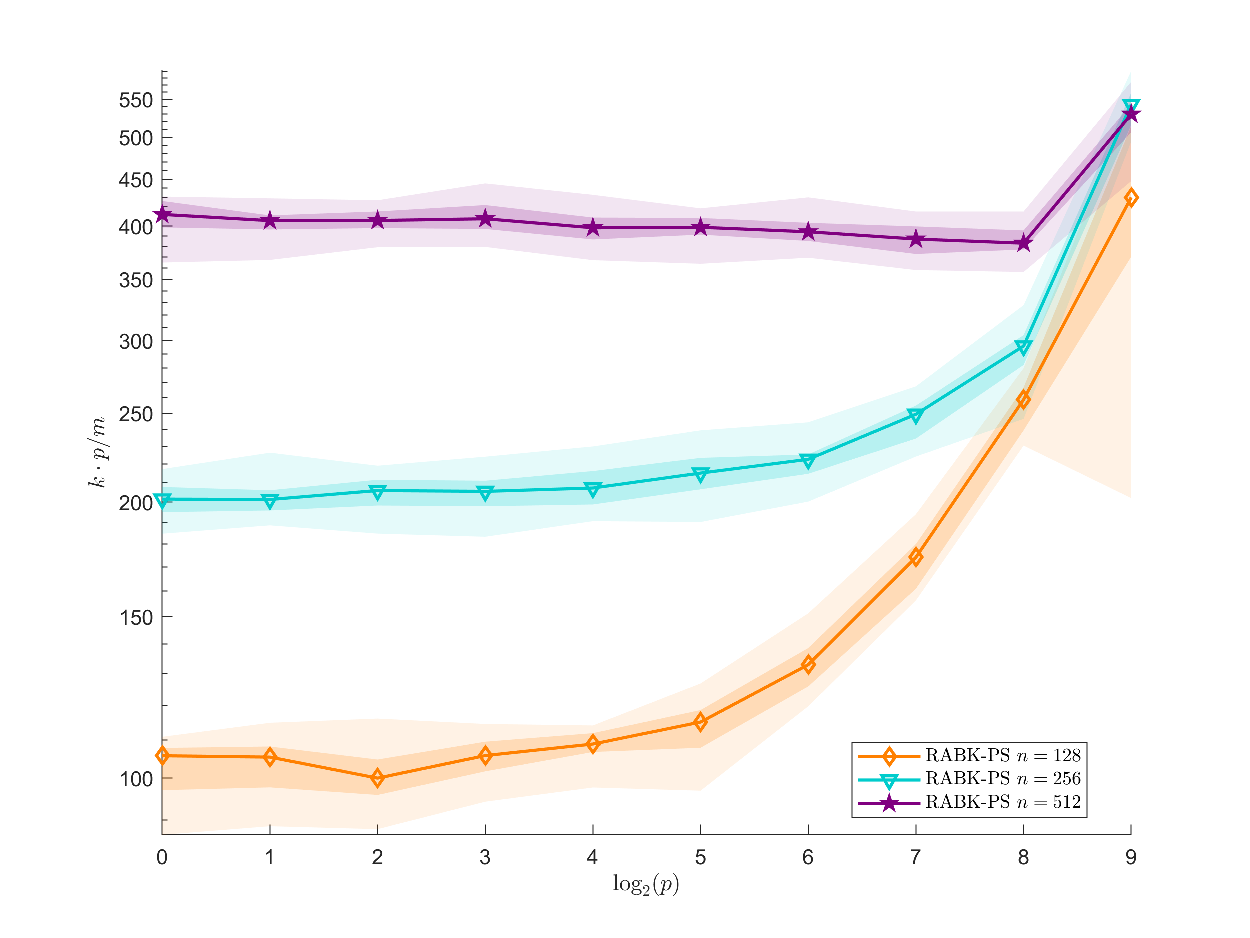}%
		}%
		
		\vspace{0.5em}
		
		\subfigure[$\kappa_A=2$, $\kappa_B=10$]{%
			\includegraphics[width=0.32\textwidth]{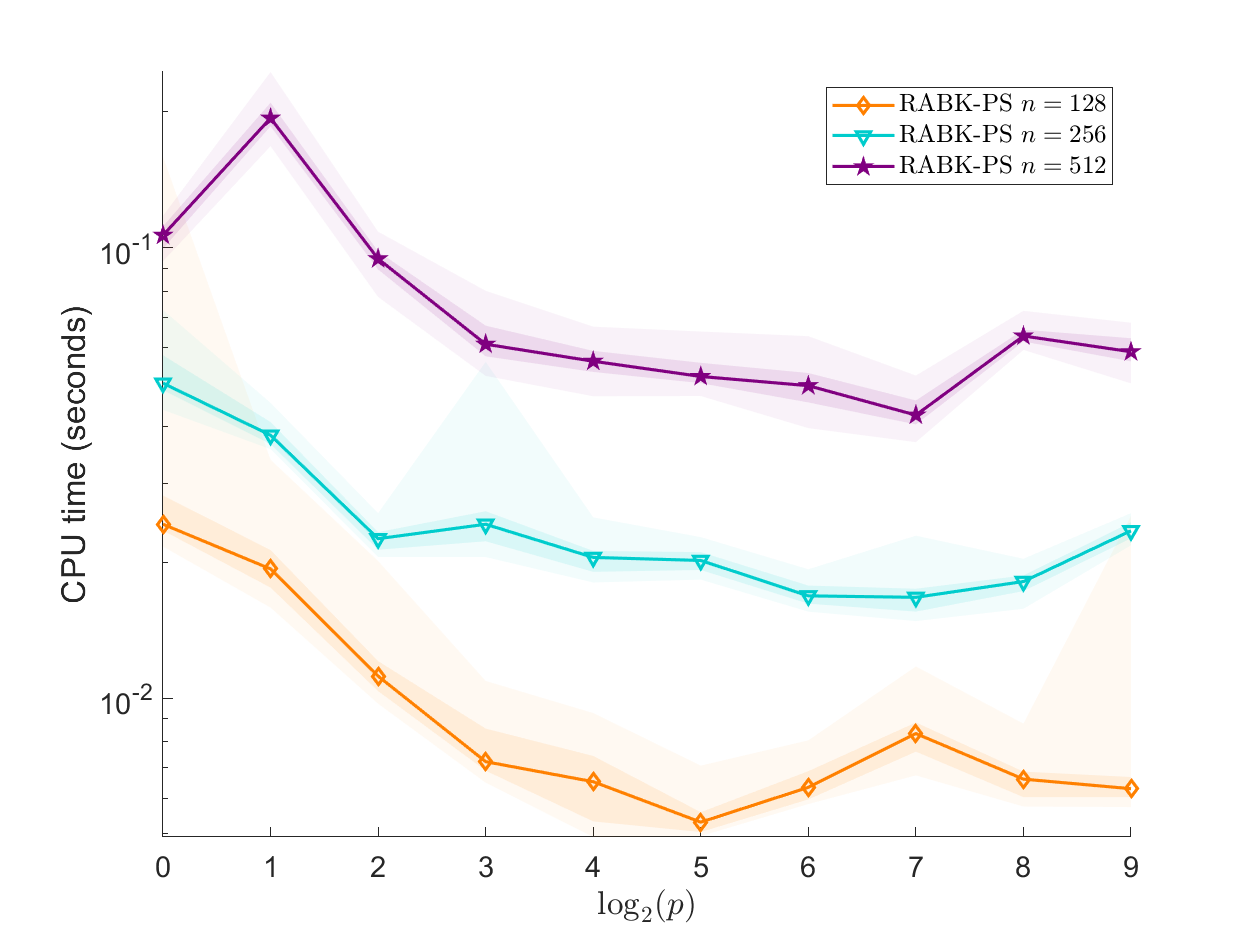}%
		}%
		\hfill
		\subfigure[$\kappa_A=10$, $\kappa_B=10$]{%
			\includegraphics[width=0.32\textwidth]{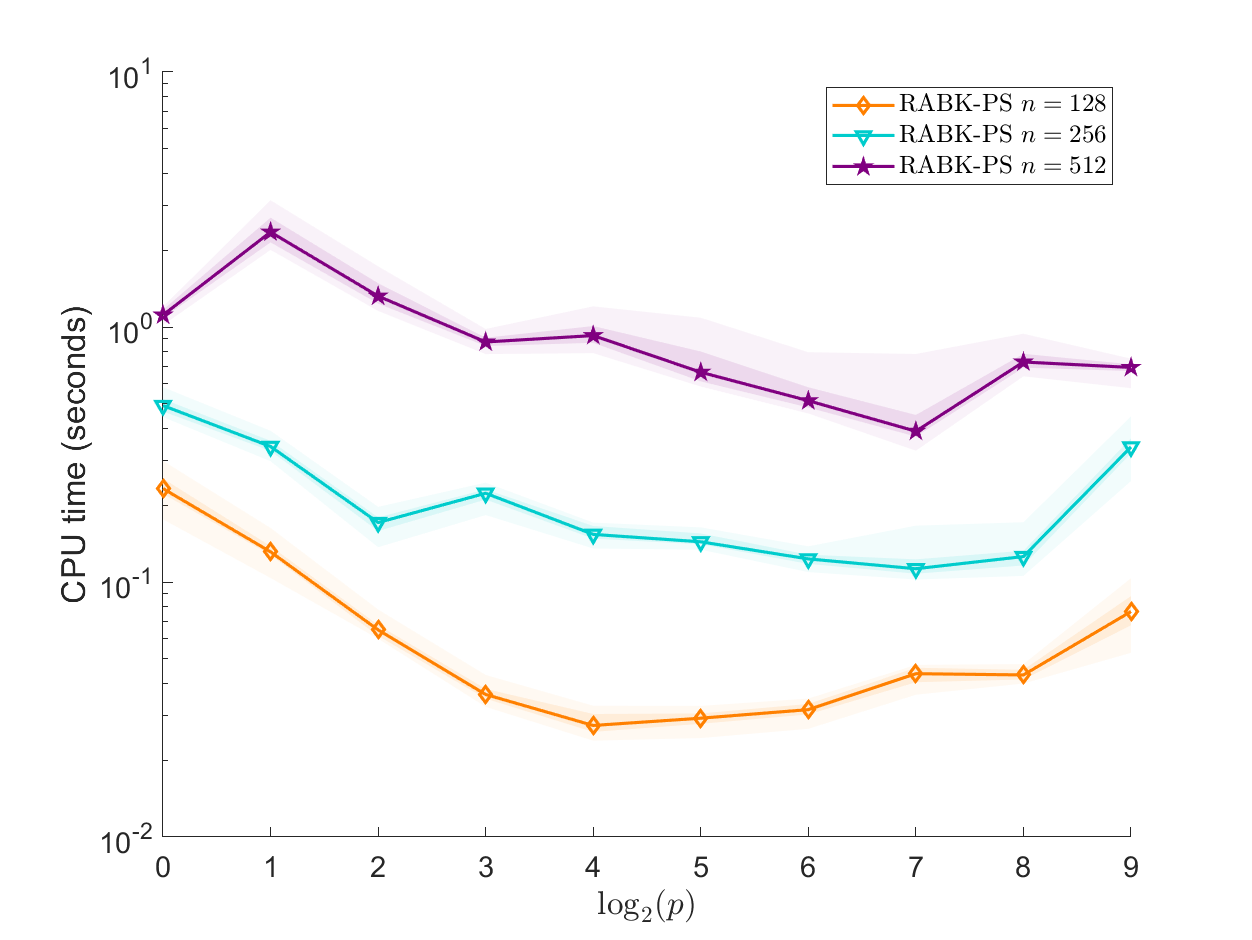}%
		}%
		\hfill
		\subfigure[$\kappa_A=10$, $\kappa_B=2$]{%
			\includegraphics[width=0.32\textwidth]{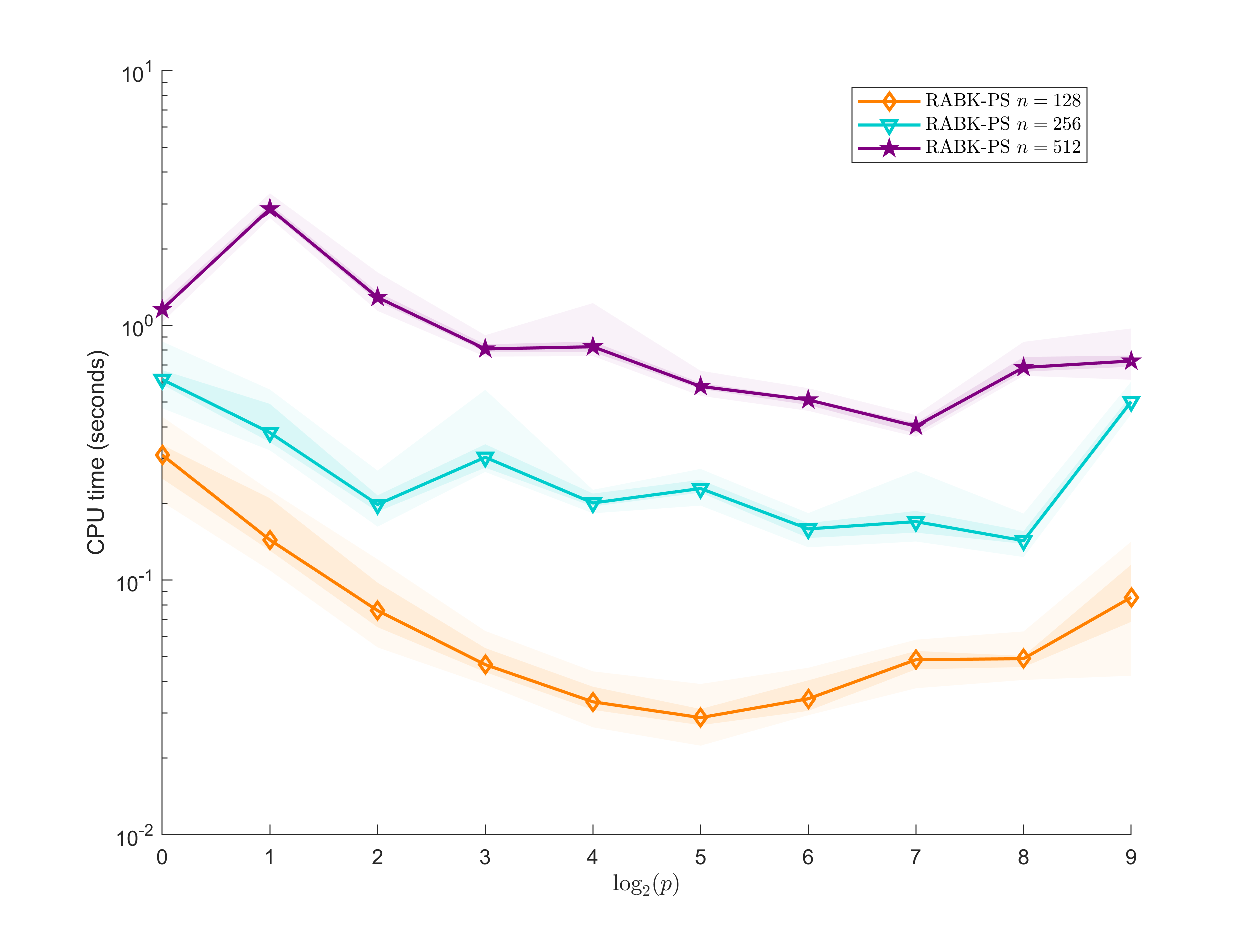}%
		}%
		
		\caption{Influence of the block size $\ell$ on RABK-PS. Top row is shown about scaled iteration count $k \cdot \ell / m$ versus $\log_{2}(\ell)$. Bottom row is shown about CPU time versus $\log_{2}(\ell)$. We set $m=512$ and $n=128, 256, 512$.}
		\label{fig:block_size_influence}
	\end{figure}

	\subsection{Effect of $\kappa_A$ on algorithmic performance}
	
	It follows from Subsection~\ref{sec:comparison_sketches}  that $\kappa_A$ has a notable impact on the convergence behavior of the algorithms. In this subsection, we further investigate this effect on the CPU time of both RIM and RIM-PS. We focus on the overdetermined case ($m > n$) as a representative setting, with $m=2000$, $n=1000$, and $\kappa_B \in \{1,5,10\}$. Based on the observation in Subsection~\ref{sec:comparison_sketches} that excessively small or large block sizes lead to slower convergence, we fix $\ell = 100$ throughout the experiments. The numerical results are displayed in Figure~\ref{fig:iter_cpu_comparison1}, where the number of iterations and the corresponding CPU time are plotted against $\kappa_A$. The top row reports the iteration counts required for convergence, while the bottom row shows the corresponding CPU time.
	
	\begin{figure}[htbp]
		\centering
		\subfigure[$\kappa_B=1$]{%
			\includegraphics[width=0.32\textwidth]{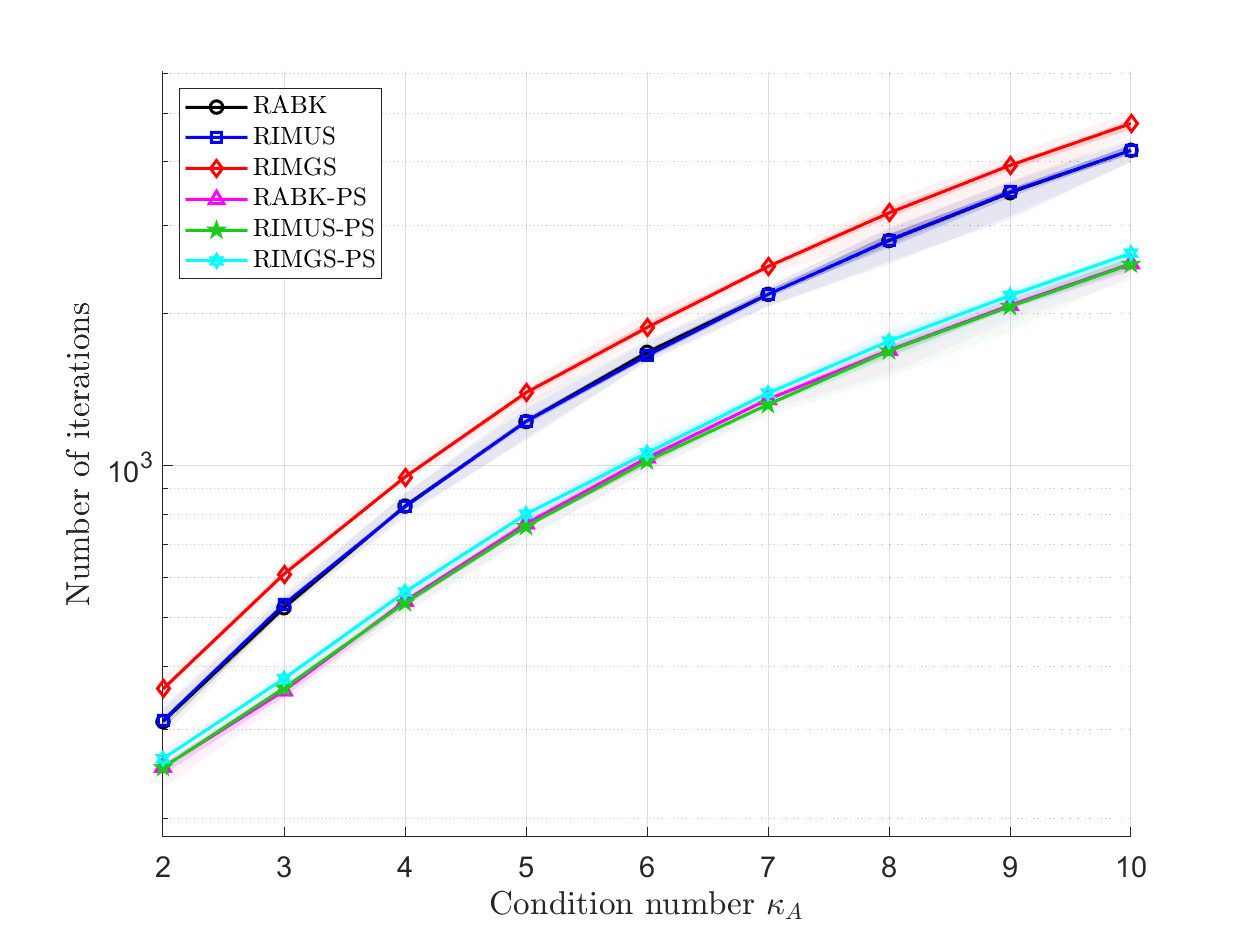}%
		}%
		\hfill
		\subfigure[$\kappa_B=5$]{%
			\includegraphics[width=0.32\textwidth]{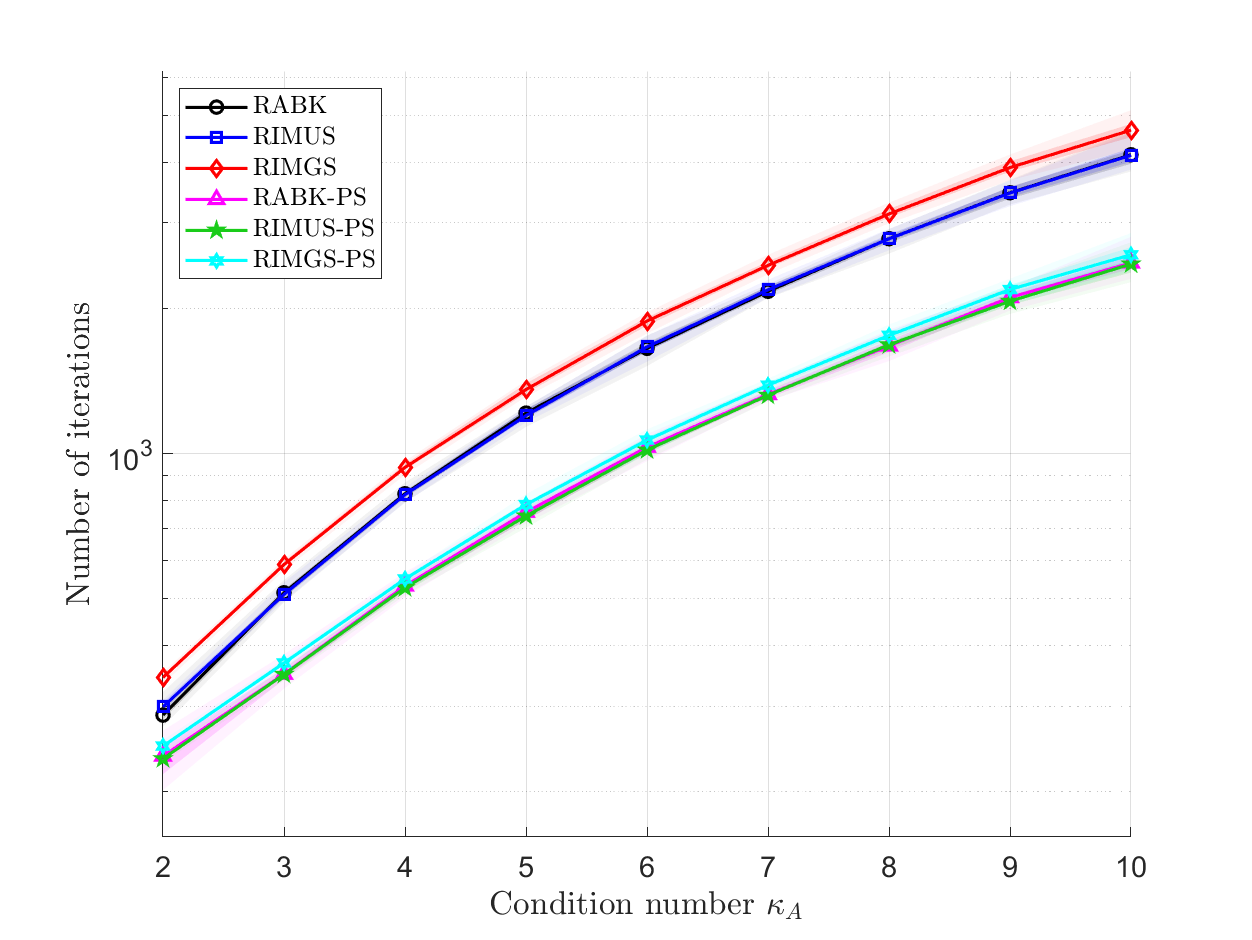}%
		}%
		\hfill
		\subfigure[$\kappa_B=10$]{%
			\includegraphics[width=0.32\textwidth]{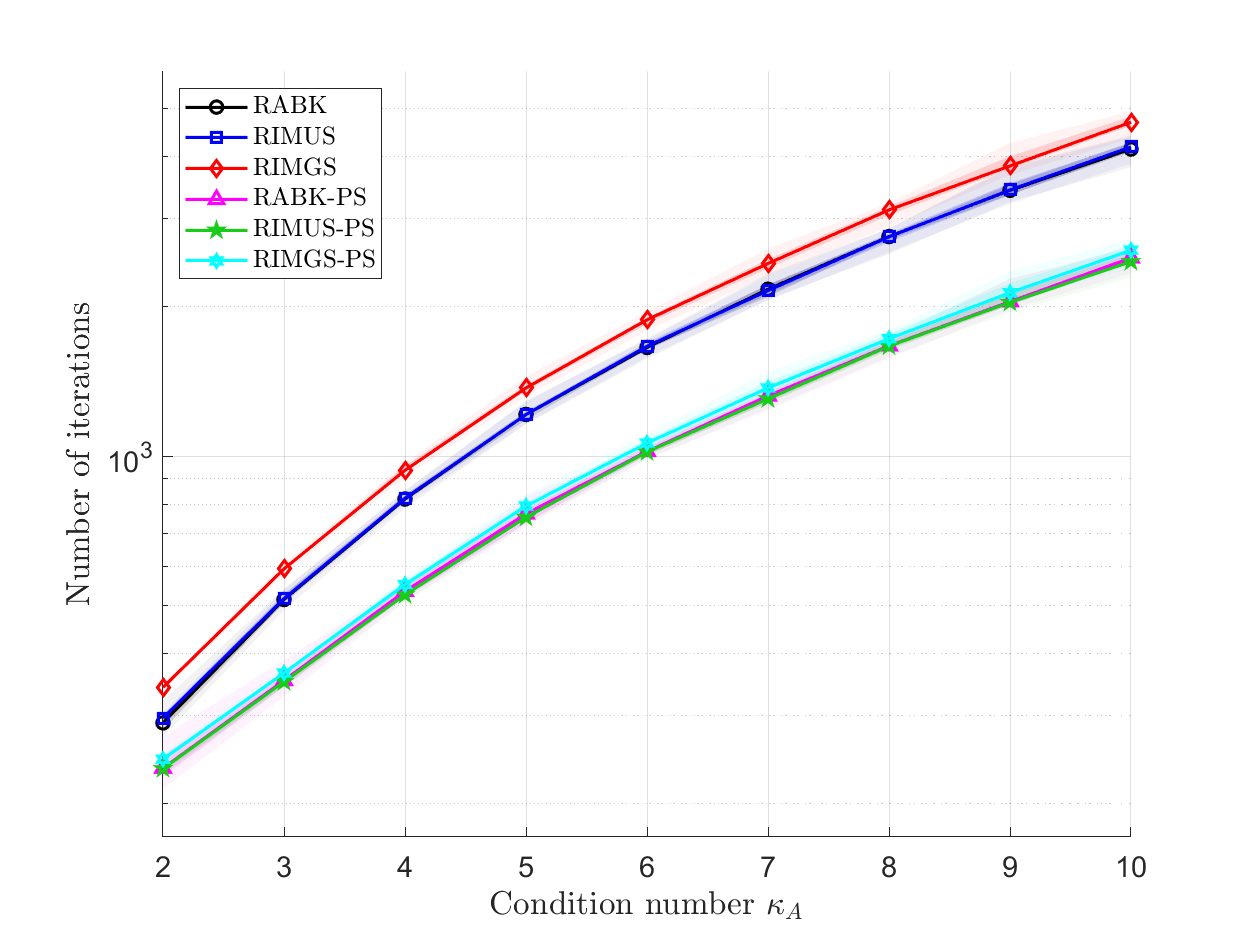}%
		}%
		
		\vspace{0.5em}
		
		\subfigure[$\kappa_B=1$]{%
			\includegraphics[width=0.32\textwidth]{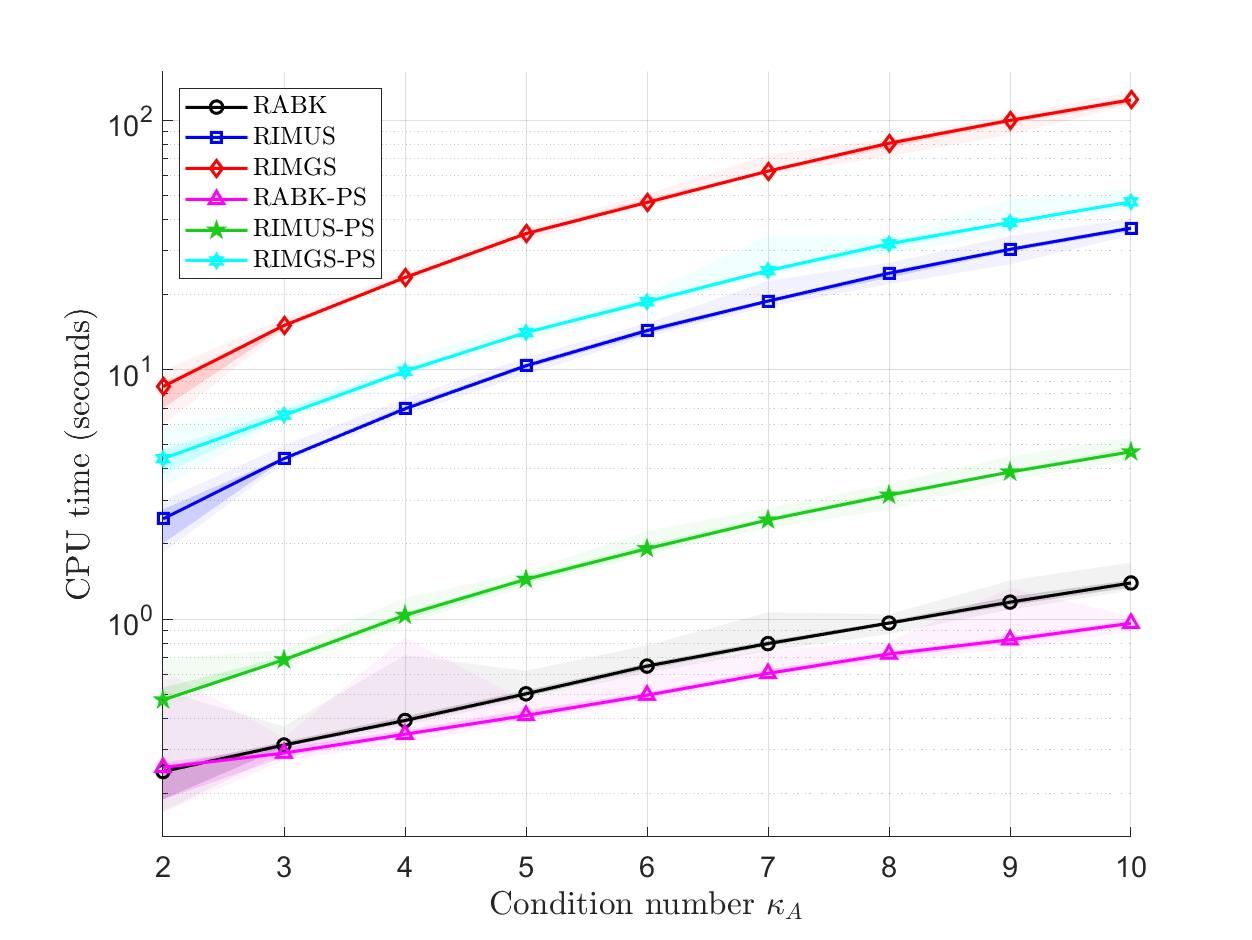}%
		}%
		\hfill
		\subfigure[$\kappa_B=5$]{%
			\includegraphics[width=0.32\textwidth]{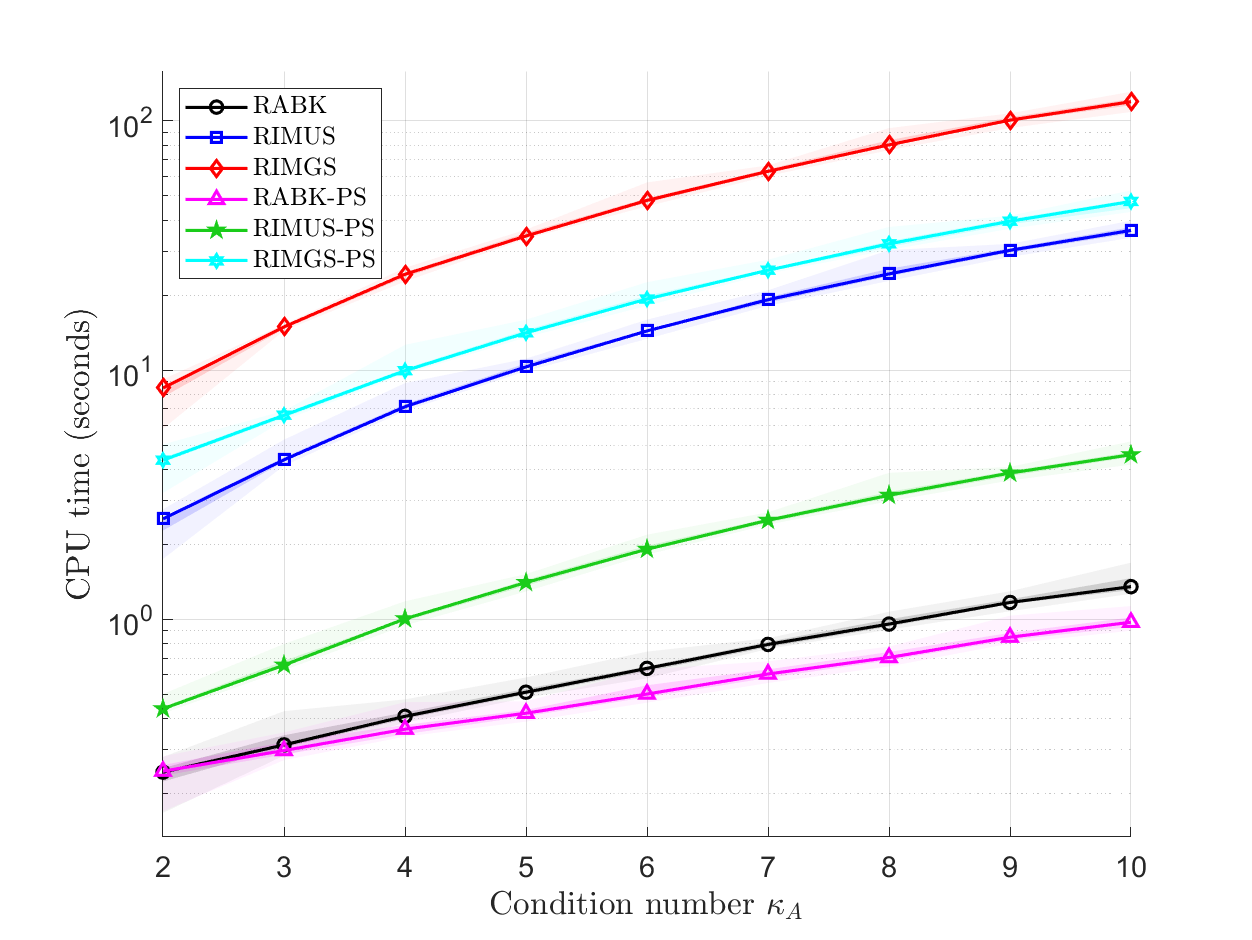}%
		}%
		\hfill
		\subfigure[$\kappa_B=10$]{%
			\includegraphics[width=0.32\textwidth]{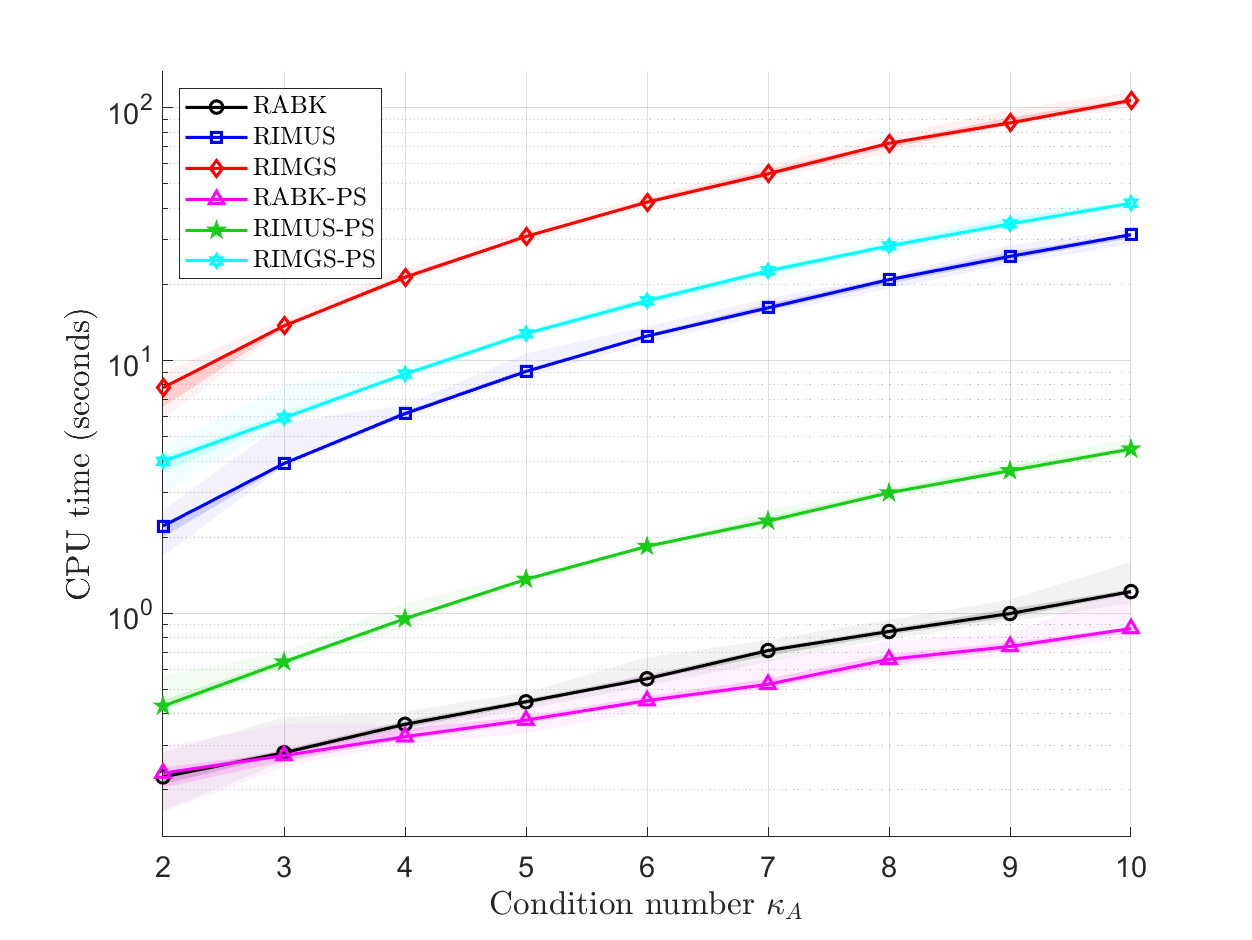}%
		}%
		
		\caption{Performance comparison RABK-PS, RIMUS-PS and RIMGS-PS with RABK, RIMUS and RIMGS for overdetermined systems under three different $\kappa_B$. We set $m=2000, n=1000, \ell=100$. Top row: Number of iterations versus $\kappa_A$. Bottom row: CPU time versus $\kappa_A$.}
		\label{fig:iter_cpu_comparison1}
	\end{figure}

	It can be seen from the top row of Figure~\ref{fig:iter_cpu_comparison1}  that RABK-PS, RIMGS-PS and RIMUS-PS, which use Polyak step-size,  require  fewer iterations to converge than their RIM counterparts. The iteration counts increase rapidly as $\kappa_A$ grows from $2$ to $10$. The bottom row shows that RABK and RABK-PS consistently outperform the other methods under the same conditions. Across all tested configurations, the Polyak step-size variants (RABK-PS, RIMUS-PS, and RIMGS-PS) achieve a reduction in CPU time compared to their RIM counterparts in nearly all cases; the only exception occurs for small $\kappa_A$, where RABK performs comparably to RABK-PS. Moreover, as $\kappa_A$ increases, the CPU time of the Polyak variants grows less sharply than that of the RIM methods. Overall, RABK-PS exhibits the fastest convergence and the lowest sensitivity to $\kappa_A$. Consequently, we select RABK-PS with a fixed block size $\ell=100$ for subsequent comparisons with other classical algorithms.
	
	\subsection{Comparison to some existing methods}
	This subsection compares the performance of RABK-PS with that of RABK, SLA, and MAP. We consider three different combinations of $\kappa_A$ and $\kappa_B$ and test these settings across varying problem scales by increasing $m$ from $1000$ to $1500$ while fixing $n=1000$. The computation is stopped when the RSE drops below the prescribed threshold $10^{-12}$.
	
	It can be seen from Figure~\ref{fig:comparison_cpu} that RABK, RABK-PS and MAP consistently outperform SLA across all test cases. In particular, RABK-PS is always the fastest algorithm. The advantage becomes more obvious as the problem scale grows, where the CPU time of SLA and MAP increase steeply with $m$, while the CPU time of RABK and RABK-PS grows at a slower rate. This is because randomized methods depend only on the selected block of rows, not on the full matrix dimensions. These results demonstrate the advantage of RIM-PS.
	
	\begin{figure}[htbp]
		\centering
		\subfigure[$n=1000$, $\kappa_A=6$, $\kappa_B=2$]{%
			\includegraphics[width=0.32\textwidth]{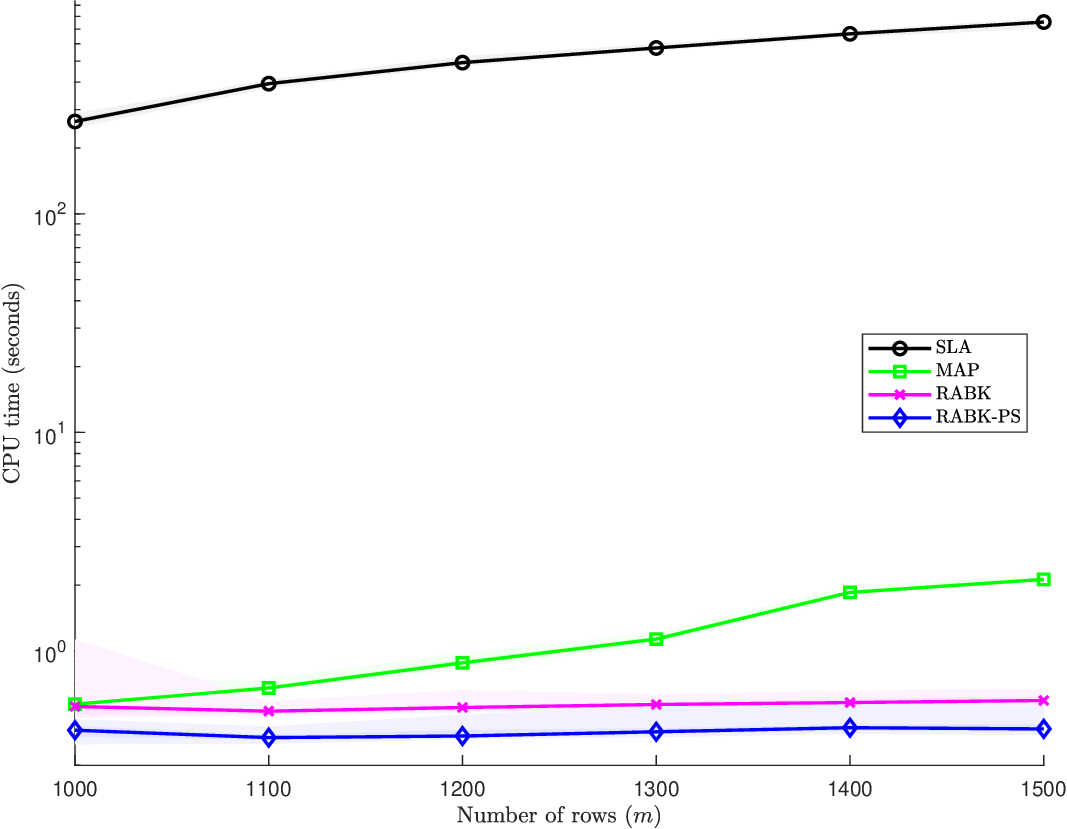}%
		}%
		\hfill
		\subfigure[$n=1000$, $\kappa_A=6$, $\kappa_B=10$]{%
			\includegraphics[width=0.32\textwidth]{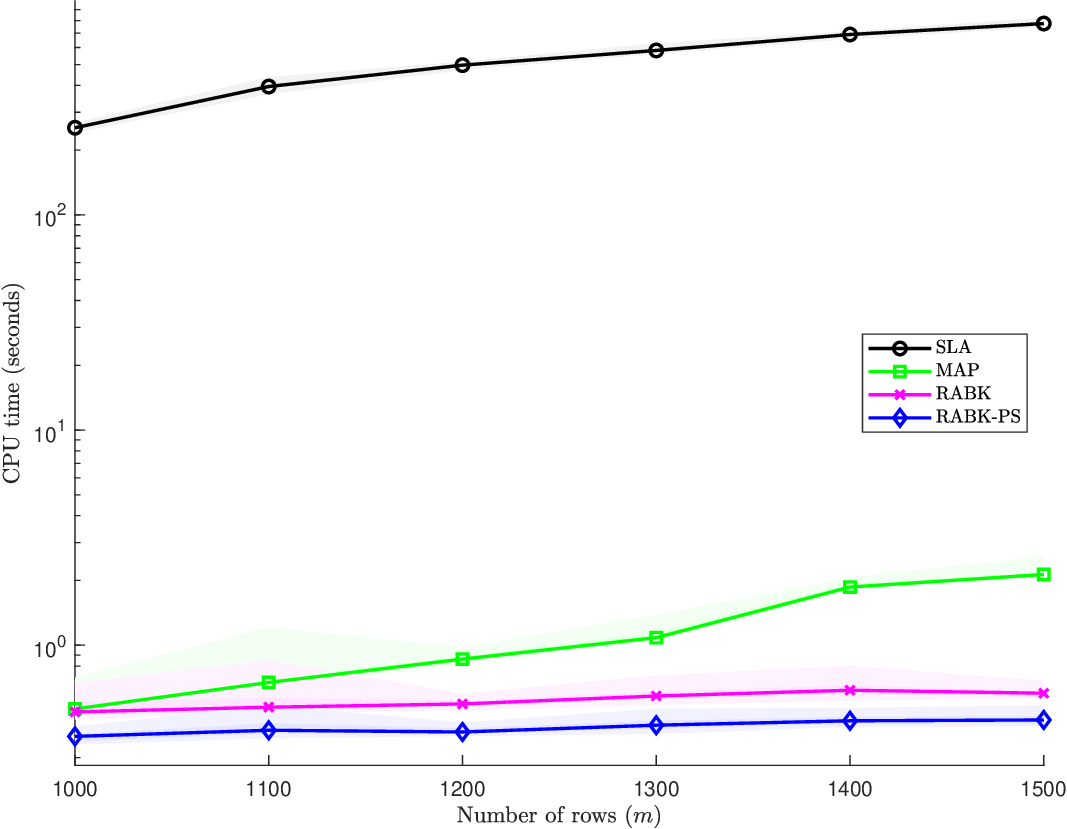}%
		}%
		\hfill
		\subfigure[$n=1000$, $\kappa_A=4$, $\kappa_B=4$]{%
			\includegraphics[width=0.32\textwidth]{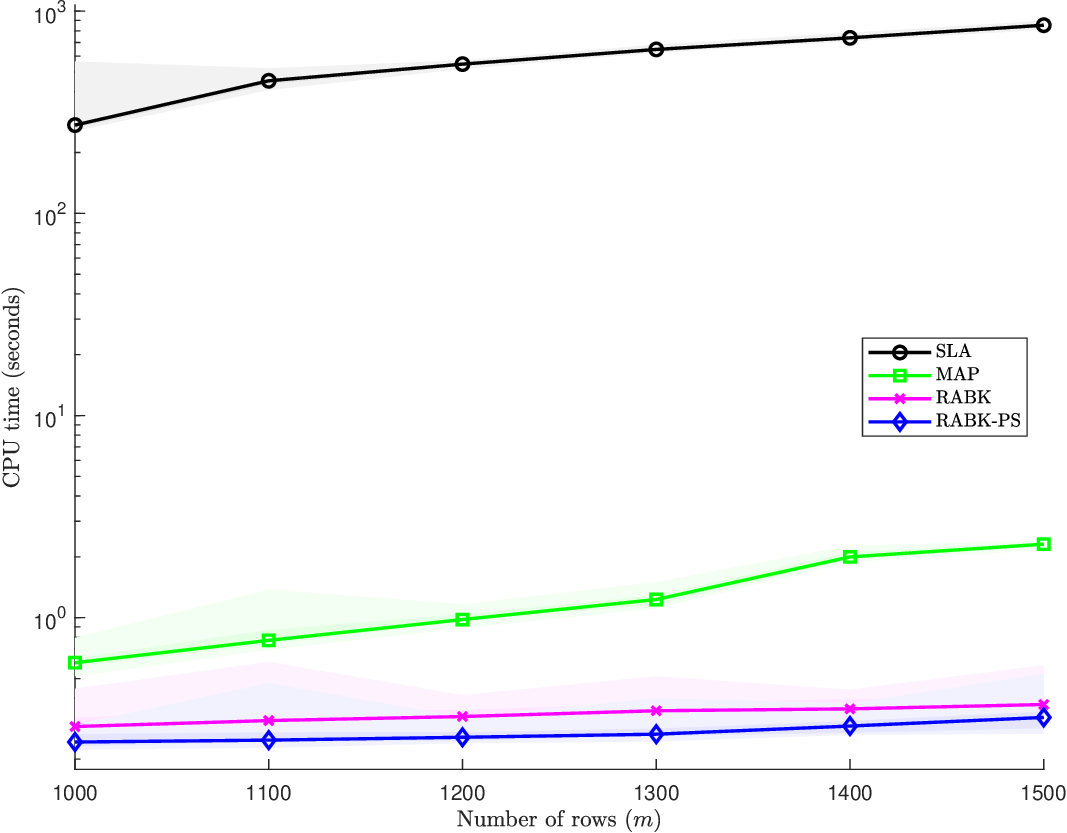}%
		}%
		\caption{CPU time comparison of SLA, MAP, RABK, and RABK-PS as the number of rows $m$ grows.}
		\label{fig:comparison_cpu}
	\end{figure}

	\section{Conclusion}\label{sec:conclusion}
	In this paper, we have studied a Polyak step-size variant of the RIM for solving the GAVE \eqref{gave}, which accommodates non-square coefficient matrices. We have established the linear convergence rate of the proposed method in expectation. Numerical experiments demonstrate that RIM-PS offers notable advantages in terms of iteration efficiency and robustness to conditioning.
	
	There are still many possible future avenues of research. For example, incorporating momentum acceleration techniques, such as Gearhart-Koshy acceleration~\cite{Rieger2023,Hegland2023} or heavy-ball momentum~\cite{Polyak1964,Ghadimi2015}, into RIM-PS may further enhance its convergence performance. Whether such combinations preserve the convergence guarantees of RIM-PS merits further investigation.

	
	
	
%
	
	\bibliography{myref}
	
\end{document}